\documentclass[11pt,captions=nooneline,DIV=14,parskip=half]{amsart}

\usepackage{amsmath,amsthm,amssymb,mathrsfs,amsfonts,verbatim,enumitem,color}
\usepackage{bbm}

\usepackage{mathabx}
\usepackage{etoolbox} 
\usepackage{bbm}
\usepackage[all,tips]{xy}
\usepackage{graphicx,ifpdf}
\ifpdf
\fi

\usepackage{geometry}
\usepackage{lmodern}

\usepackage{scrextend}
\KOMAoption{fontsize}{13pt}

\usepackage{blindtext}

\newtheorem{thm}{Theorem}[section]
\newtheorem{lem}[thm]{Lemma}
\newtheorem{cor}[thm]{Corollary}
\newtheorem{prop}[thm]{Proposition}

\theoremstyle{definition}

\newtheorem{ex}[thm]{Example}

\theoremstyle{remark}

\numberwithin{equation}{section}

\newcommand {\new}[1] {\textcolor{black}{#1}}

\makeatletter
\newcommand{\rmnum}[1]{\romannumeral #1}
\newcommand{\Rmnum}[1]{\expandafter\@slowromancap\romannumeral #1@}
\makeatother

\newcommand{\ba}{\mathbf{a}}

\newcommand{\bx}{\mathbf{x}}

\newcommand{\R}{\mathbb{R}}
\newcommand{\N}{\mathbb{N}}
\newcommand{\SL}{\operatorname{SL}}
\newcommand{\GL}{\operatorname{GL}}
\newcommand{\Z}{\mathbb{Z}}
\newcommand{\Q}{\mathbb{Q}}
\newcommand{\C}{\mathbb{C}}

\newcommand{\til}{\widetilde}

\newcommand{\dd}{\; \mathrm{d}}
\newcommand{\ap}{\mathfrak a_{\mathfrak u}}

\newcommand{\Ad}{\operatorname{Ad}}
\newcommand{\rk}{\mathrm{rank\,}}
\newcommand{\supp}{\mathrm{supp\,}}
\newcommand{\diag}{\mathrm{diag}}

\begin{document}

\title{Expanding cone and applications to homogeneous dynamics}


\author{Ronggang Shi}
	\address{Shanghai Center for Mathematical Sciences, Fudan University, Shanghai 200433, PR China}
\email{ronggang@fudan.edu.cn}
\thanks{The author is supported by   NSFC (11871158) and 
NSF 0932078 000.}



\date{}



\begin{abstract}
Let $U$ be a  horospherical subgroup of a  noncompact simple   Lie group $H$
and let $A$ be a maximal split torus in the normalizer of   $U$.
We define the  expanding cone $A_U^+$ in  $A$ with respect to
$U$ and show that it can be explicitly calculated.
We prove several dynamical results for 
 translations of $U$-slices by elements of $A_U^+$ on finite volume homogeneous space
$G/\Gamma$ where   $G$ is a Lie group containing $H$. 
More precisely, we prove quantitative nonescape of mass and 
equidistribution of a $U$-slice.
If   $H$  is a normal subgroup of $G$ and the $H$ action on $G/\Gamma$ has a spectral gap, 
we prove  effective multiple equidistribution and  pointwise equidistribution   with an error rate. In the  paper we formulate the notion of the  expanding cone 
and prove the  dynamical results above  in the  more general  setting where  $H$ is 
 a  semisimple Lie group without compact factors. 
 
 In the appendix, joint with Rene R\" uhr, we prove a multiple ergodic theorem with an error rate. 
\end{abstract}

\maketitle
\markright{}

\section{Introduction}\label{sec;introduction}

Let $H$ be a connected semisimple Lie group without compact factors.
 Let $U$ be the unipotent radical of an absolutely proper  parabolic subgroup $P$ of $H$,
 where absolutely proper means
the projection of $P$ to each simple factor of $H$ is not surjective. 
Let 
  $A$ be a maximal 
connected $\Ad$-diagonalizable subgroup of $P$, i.e.~the image of $A$ under the adjoint representation 
is diagonal under some basis of the Lie algebra of $H$ and $A$ is maximal among all connected
subgroups of $P$ with this property. 
In this paper we define the  concept of the  expanding cone (denoted by $A_U^+$) in 
$A$
 associated to $U$ and 
prove some  dynamical results  for its translations of a $U$-slice 
(a piece of a $U$-orbit) on
finite volume  homogeneous spaces $G/\Gamma$ where
$G$ is a  Lie group containing $H$  and $\Gamma $ is a lattice  (a discrete subgroup of $G$
with finite covolume). 
 
The expanding cone is a generalization of the following example (see \S\ref{sec;cone verify} for the proof) which is often used
in homogeneous dynamics due to its close connection    to metric number theory, 
see e.g. ~\cite{klw04}\cite{km98}\cite{kw08}\cite{s09}\cite{s10}. 
\begin{ex}\label{ex;main}
Let $m, n$ be positive integers.
\begin{align*}
H=\SL_{m+n}(\R), 
 U =\left \{
 \left(
 \begin{array}{cc}
 1_m & p\\
 0_{mn} & 1_n 
 \end{array}
 \right)\in H: p\in \mathrm{M}_{mn}(\R)
 \right\} ,
\end{align*}
 $A$ is the identity  component  of the group of  diagonal matrices in $G$,  and 
\begin{align}\label{eq;cone example}
A^+_U=\{\diag(e^{r_1 }, \ldots, e^{r_m}, e^{-t_1}, \ldots, e^{-t_n})\in G: r_i, t_j>0\}.
\end{align}
\end{ex}

The expanding cone allows us to generalize several dynamical  results in the setting of Example \ref{ex;main}
to general homogeneous spaces. In this paper we  prove 
quantitative nonesape of mass, equidistribution of translates of $U$-slices, 
 effective multiple  equidistribution of smooth  $U$-slices  and    pointwise equidistribution
 with an error rate.  
 
 Besides generalizations, the novelty of the paper is that it further develops the method 
 of Kleinbock-Margulis \cite{km12} to prove effective {\em multiple} correlations of translations of  smooth measures on $U$-slices. Here multiple means 
 $k$-step for arbitrary positive integer $k$, while the previous results \cite{km12}\cite{dkl} are $1$-step and \cite{ksw}
 is $2$-step. The effective multiple correlations  have more applications 
 than the $2$-step case. 
 For example, 
   a multiple Birkhoff type ergodic theorem  is proved  in the appendix 
 joint with Rene R\"uhr and a central limit theorem is proved  in \cite{shi}.
 We remark here that it is not clear how to prove these results using the $2$-step case  even 
 in the setting of Example \ref{ex;main}. 
 We will  discuss briefly the history and the current status of the multiple ergodic theorem
 in the appendix.

\subsection{Expanding cone}\label{sec;a cone}

Let $\rho: H\to \GL(V)$ be 
a  representation of $ H $ on a (nonzero)  finite dimensional real vector space $V$. 
We say $(\rho, V)$ is irreducible if $\{0\}$ and $V$
are the only $ H $-invariant subspaces. 
 For every $a\in A$ we can write  
\begin{align*}
V=V^+_a\oplus V^0_a\oplus V^-_a
\end{align*}
where $V^+_a, V^0_a, V^-_a$ are the sums of the eigenspaces of $\rho(a)$ with eigenvalues $>,=, <1$,
respectively. 
Following  \cite{s}
we say $ U $ is 
$a$-expanding\footnote{The concept  ``$U$ is $a$-expanding" is not exactly the same as that in \cite{s}, but they
are equivalent which can be deduced from \cite[Lemma A.1]{s}. } for some $a\in A$ if, for every nontrivial irreducible representation $(\rho, V)$ of $ H $,  the subspace
\begin{align*}
V^{ U }:=\{v\in V: \rho(h)v=v \quad \forall h\in U\}
\end{align*}
 is contained in $V_a^+$. 
The expanding cone of $ U $ in $A$ is defined   to be 
\begin{align}\label{eq;expanding cone}
A^+_U:=\{a\in A:  U  \mbox{ is } a\mbox{-expanding}\}. 
\end{align}

Let $\mathfrak h$ and $\mathfrak u$ be the Lie algebras of $H$ and $U$, respectively. 
We will show in Theorem \ref{thm;open} that the expanding   cone $A_U^+$ only depends on the Lie algebras $\mathfrak h, \mathfrak u$ and 
it can be described explicitly. 
To do this we need to introduce some notation related to  Lie algebras. 
Let $\mathfrak a$  be the Lie algebra of $A$ 
 and let $\mathfrak a^*$ be the dual vector space of $\mathfrak a$. 
Let  $\Phi=\Phi( \mathfrak h,\mathfrak a)\subset \mathfrak a^*$  be the relative root system of $\mathfrak h$ and let $\Phi(\mathfrak u)$ be the subset of
roots whose eigenvectors are in $\mathfrak u$.
 Recall that the  Killing form $\mathbf B(\cdot, \cdot)$ is positive definite on $\mathfrak a$, so for each $\alpha\in \mathfrak a^*$
 we can associate an $ s _\alpha\in \mathfrak a$ by $\mathbf B( s _\alpha,  s )=\alpha( s )$ for every $ s \in \mathfrak a$.
 By this isomorphism we can associate to $\Phi(\mathfrak u)$ the following open and  convex cone in $\mathfrak a$:
 \begin{align}\label{eq;cone}
\mathfrak a^+_{\mathfrak u}:=
 \big\{
 \sum_{\alpha\in \Phi(\mathfrak u)} t_\alpha  s _\alpha: t_\alpha> 0\ \mbox{for all }\alpha
\big\}. 
\end{align}
Let $\exp : \mathfrak a\to A$ be the usual exponential map which 
is  an isomorphism 
from  the additive group $\mathfrak a$ to  $A$. 
\begin{thm}\label{thm;open}
The expanding cone $A^+_U=\exp \mathfrak a^+_{\mathfrak u}=\{\exp  s :  s  \in \mathfrak a^+_{\mathfrak u}\}$.  
\end{thm}

There are  possible  refinements of the  expanding cone which we will not deal with in this paper: 
one can 
  define the expanding cone 
in any $\Ad$-diagonalizable group which normalizes an $\Ad$-unipotent subgroup; or  one can define the  expanding cone with respect to a fixed 
representation of $ H $. 
These generalisations may also have applications in homogeneous dynamics.

\subsection{Quantitative nonescape of mass}\label{sec;nonescape}
 In the setting of Example \ref{ex;main}, 
quantitative nonescape of mass for the $A_U^+$ action on $ H /SL_{m+n}(\Z)$
 was first studied  by Kleinbock  and Margulis \cite{km98}. 
Apart from  its applications in metric number theory, it is a useful tool  while studying 
equidistribution of measures on homogeneous space with respect to the action of diagonal elements. There are various generalizations of \cite{km98} and  most of them 
are in the setting of Example \ref{ex;main}, see e.g. \cite{klw04}, \cite{kw08} and \cite{dfsu}.
Using the  expanding cone we 
establish  quantitative nonescape of mass on general  homogeneous spaces for $ U $-slices. 

In this paper,  besides compatible  group structure and manifold structure we assume 
 the  Lie group  $G$ has two more structures, namely,  a right Haar measure $\mu_G$
 and 
 a right invariant Riemannian structure which induces a
 right invariant metric $d_G$. 
Different choices of these two structures will only affect some constants, so we may assume 
the compatibility among several groups  if needed.
 We use $B_r^G$ to denote the open  ball of radius $r$  in $G$
 centered at the identity element $1_G$.
For a lattice   $\Gamma $ of $G$ we use 
   $d_X$
to denote  the 
 induced metric  on
$X=G/\Gamma $ 
defined by 
\begin{align*}
d_X(g_1\Gamma , g_2\Gamma )=\inf_{\gamma\in \Gamma } d_G(g_1\gamma, g_2)\quad\mbox{where}\quad g_1, g_2\in G. 
\end{align*}
For every $x\in X$ we let $\pi_x: G\to X$ be the map defined by $g\to gx$. 
The injectivity radius  at $x\in X$ is defined by 
\begin{align*}
I(x, X)=\sup \{r\ge 0: \pi_x \mbox { is injective on } B_r^G\}
\end{align*}
which  measures the closeness  of  $x$ to $\infty$.
For every $\varepsilon >0$ the complement of the set
\begin{align}\label{eq;injective region}
\mathrm{Inj}_\varepsilon^X:=\{x\in X:  I(x, X)\ge \varepsilon\}
\end{align}
can be thought of the  $\varepsilon$-neighborhood of $\infty$ in $X$. 
The following simple observation may justify this  point of view: a sequence 
$\{x_i\}_{i\in \N}$ in X diverges if and only if for any $\varepsilon >0$ one has
$x_i \not \in \mathrm{Inj}_\varepsilon^X$ for all   sufficiently large $i$.

 

\begin{thm}\label{thm;general}
Let $G$ be a connected semisimple   Lie group containing 
 $ H $,  and let 
 $\Gamma $ be a lattice of  $G$. Then there exists $\delta=\delta(G, \Gamma, H)>0$ such that 
  for every compact subset $L$ of $ X=G/\Gamma $,
    for every $a\in \overline{A^+_U}$,  $x\in L$ and $r>0$ one has
\begin{align}\label{eq;better}
\mu_U(\{h\in B_r^{ U }: a hx\not\in \mathrm{Inj}_\varepsilon^X\}) \ll_{L,r}  \varepsilon ^\delta.
\end{align}  
\end{thm}
As usual the relation    $f_1\ll_* f_2$ for two nonnegative  functions   means $f_1 \le M f_2$ for some constant $M>0$ depending possibly  on variables in the set $*$. 
In this paper 
the dependence of the implied constants on metrics or Haar measures of Lie groups will not be specified. 
 Theorem \ref{thm;general} is an important ingredient for effective 
 multiple 
 equidistribution in \S \ref{sec;intro k mixing}.

It is interesting  to extend Theorem \ref{thm;general} to more general classes of measures on $U$ such as friendly measures  of 
\cite{klw04}  in
the setting of Example \ref{ex;main}. Although Theorem \ref{thm;general} looks restrictive due to 
the assumption that $G$ is connected and  semisimple, it actually provides  quantitative nonescape of mass
for an arbitrary  Lie group $S$ and a lattice $\Lambda$ in $S$. 
It is not hard to see that if $S$ is not connected, the quantitative nonescape of mass 
can be reduced to that of  $S^\circ/S^\circ \cap \Lambda$. 
Here and hereafter $S^\circ$ denotes the identity component of the Lie group. 
Suppose  $S$ is  connected, then  it follows from \cite[Lemma 6.1]{bq12} that 
there is a closed normal amenable subgroup $R$
of $S$ such that $S/R$ is a semisimple Lie group  without compact factors,  $R/\Lambda\cap R $ is compact
and $\pi(\Lambda)$ is a lattice of $S/R$ where $\pi: S\to S/R$ is the quotient map. 
Therefore, we can finally reduce quantitative nonescape of mass to the setting of Theorem \ref{thm;general}. 

\subsection{Translates of $ U $-slices  by elements of $A_U^+$}\label{sec;intro translate}
Let $\Gamma $ be a lattice of a Lie group $G$ containing  $H$ and let $X=G/\Gamma $. 
For every $x\in X$ 
it 
follows from Ratner's theorem \cite{r91},  which was conjectured by  Raghunathan,  that 
$\overline{Hx}$ (the closure of $Hx$ in $X$) is homogeneous, i.e. 
$\overline{Hx}=Sx$ for some closed subgroup $S$ of $G$, and there is  an $S$-invariant probability measure 
$\mu_{\overline{Hx}}$ supported on $\overline{Hx}$.  
  We say a sequence $\{a_i\}_{i\in \N}$ in $A_U^+$ drifts away from walls
if   ${d}_A( a_i, \partial A_U^+)\to \infty $
as $i\to \infty$, where $\partial A_U^+$ is the boundary of $A_U^+$ in $A$.

\begin{thm}\label{thm;translate}
Let $G$ be a Lie group  containing  $ H$, let  $\Gamma $ be a lattice of $G$,   and let $X=G/\Gamma $. 
 Let $\{a_i\}_{i\in \N}$ be a sequence 
in $A_U^+$ drifting away from  walls. Then for any $f\in L^1( U ), \psi \in C_c(X)$\footnotemark
\  and $x\in X$ one has
\begin{align}\label{eq;translate}
\lim_{i\to \infty}\int_{ U } f(h)\psi (a_ihx) \dd\mu_U (h)= \int_{ U }f \dd\mu_U\cdot \int_X\psi  \dd\mu_{\overline {Hx}}.
\end{align}
\end{thm}

\footnotetext{
We use $C_c(Y)$ and $C^\infty(Y)$ to denote the real valued compactly supported  and smooth functions
on the space $Y$, respectively. The notion $L^p(Y)\ (p>0)$  is the  usual complex valued $L^p$-space with respect to certain fixed volume measure on $Y$.}
Certain  cases of Theorem \ref{thm;translate} were obtained earlier  by Mohammadi and Salehi-Golsefidy \cite{ms}.
More precisely, for $H=\mathbf H(\R)^\circ$ where $\mathbf H$ is a simply connected and 
$\Q$-simple algebraic group 
with $\rk_{\Q}\mathbf H=\rk_{\R}\mathbf H$, 
\cite{ms} proves Theorem \ref{thm;translate} in the 
case where $G=H$, $\Gamma=\mathbf H(\Z)\cap H$,   $U$ is the unipotent radical of 
a minimal  parabolic subgroup defined over $\Q$, $x\in X$ with $Ux$ closed and $a_i\ (i\in \N)$ belong to a maximal $\Q$-split torus. In \S \ref{sec;intersect}, we explain how Theorem \ref{thm;translate} applies to  this case  without assuming $\rk_{\Q}\mathbf H=\rk_{\R}\mathbf H$.

Our proof of Theorem \ref{thm;translate},
which is given at the end of  \S \ref{sec;property},  is based on the property of the  expanding cone and 
Shah-Weiss  \cite[Theorem 1.4]{sw96}.
 The main new point   here  is that our cone $A_U^+$ is global (not depending on the adjoint representation of $H$ on the Lie algebra of $G$) and is  explicitly described.

If we assume in addition that conjugation by the   $a_i\ (i\in \N)$ expands a horospherical subgroup of $H$ contained in $U$, then  we can  make
(\ref{eq;translate}) effective.
 This will be the topic of \S \ref{sec;intro k mixing}. 
 The properties  of the expanding cone are used in different ways in effective and ineffective results. 
 The proof of Theorem \ref{thm;translate} uses Ratner's theorem on measure classification of 
 $U$-invariant probability measures and the expanding cone helps us to avoid other $U$-invariant and ergodic probability  measures in
 the limit. In the effective case the rate of convergence comes from the  spectral gap, and the  expanding 
 cone is used to control the behaviour of the measure near the cusp through 
quantitative nonescape of mass in Theorem \ref{thm;general}.

\subsection{Effective multiple  equidistribution}\label{sec;intro k mixing}
Suppose 
 $ H $ is  a normal subgroup of a connected
 semisimple Lie group   $G$. Let    $\Gamma$ be a   lattice of  $G$ and 
  $\mu_X$ be 
the probability Haar measure on $X=G/\Gamma$.
 We assume that the action of $ H $ on $X$ has a spectral gap, 
 i.e.~there is a compactly supported probability measure $\nu$
 on $H$ and $\delta >0$ such that 
 \[
\left \| \int_H  \psi (g^{-1} x ) \dd\nu(g)\right \|_{L^2(X)}\le (1-\delta)\|\psi \|  _{L^2(X)}
 \]
  for any 
 $\psi\in L^2_0(X)
:=\{\psi \in L^2(X): \int_X\psi \dd\mu_X=0\}$.  
Recall that a lattice  $\Gamma$ is said 
to be irreducible if,  for any connected noncompact normal subgroup $N$ of $G$, the projection of $\Gamma$ to the  quotient group $G/N$ is dense. 
 If $G$ has no compact factors
and  $\Gamma$ is irreducible,
 then the action of $ H $ on $X$ has a spectral gap, see \cite{ks}. 
 Suppose  
$G$ has finite center and there is    an almost direct product  decomposition  $G=G_1\cdots G_k$ by connected normal subgroups $G_i$ such that 
$\Gamma\cap G_i$ is a lattice of $G_i$. 
Then  the action of $H$ on $X$ has a spectral gap
if and only if 
the action of $(H\cap G_i)^{\circ}$ on  the quotient 
$G_i/\Gamma\cap G_i$ has a spectral gap for every $i$.

Let $\mathcal S(X)=\R +C_c^\infty(X)$ be the space of real valued  smooth functions on $X$ which can 
be written as the sum of a compactly supported function and a constant function. 
The aim of   this section is  to  give an effective  estimate of  the integral 
\begin{align*}
I_{f}({\bx}; \ba; \Psi)
=\int_{ U } f(h) \prod_{i=1}^k \psi _i( a_i \cdots a_1h x_i)\dd \mu_U (h)
\end{align*}
where 
\begin{align}
&f\in  C_c^\infty ( U )
 &&  \Psi=(\psi _1, \ldots, \psi _k)\in 
\mathcal S(X) ^k,
\label{eq;notation}\\
 &\ba=(a_1,\ldots, a_k)\in A^k,  && \bx=(x_1, \ldots, x_k)\in X^k.\label{eq;notation1}
\end{align}

 Let  $\{ b_t : t\in \R\}$ be
 a one parameter   subgroup  of $A$ such that the projection 
of $b=b_1$ to each simple factor of $H$ is not neutral element and
the unstable horospherical subgroup of $b$
\[
H _{b}^+:=\{h\in H: b^{-n} hb^{n}\to 1_H \quad \mbox{as }n\to \infty\}\le U.
\]  In particular 
this implies that  the positive ray $\{ b_t : t> 0\}$ is contained in $A_U^+$ (see \cite[Lemma 5.2]{s96}).
 For $a\in A$ and $\ba$ as in (\ref{eq;notation1}) we set 
\begin{align}
\lfloor a\rfloor&= \sup \{t\ge 0:  ab_{-t}\in A_U^+, H_b^+\le  H ^+_{ab_{-t}}\},  
\label{eq;a floor}\\
\lfloor \ba \rfloor&=\min \{\lfloor a_i\rfloor: {1\le i\le k}\}\notag.
\end{align}
In the definition of (\ref{eq;a floor}) if  the set before taking the supremum is empty 
 we interpret
$\lfloor a \rfloor=0$.  In view of  $b\in A_U^+$, we must 
have $a\in A_U^+$ and $H_b^+\le H_a^+$ if     $\lfloor a\rfloor>0$.  

\begin{thm}\label{thm;effective}
Let  $H$ be a normal subgroup  of a connected semisimple Lie group $G$, let $\Gamma$ be a  lattice  of $G$, and let $X=G/\Gamma$.
Suppose that the action of $ H $ on $X$ has a spectral gap.  
Let $\{b_t\}$ be a
 one parameter subgroup 
 of $A$  such that the projection 
of $b=b_1$ to each simple factor of $H$ is not the identity element and
$ H _{b}^+\le U$.
There exists $\delta=\delta( H, X,   \{b_t\} )>0$ with the following properties:
For every positive integer $k$,  any compact subset $L$ of $ X $ and any 
$f, \Psi $ as in (\ref{eq;notation}),
  there exists $M=M(L, f, \Psi , k)>0$
such that for any $\bx, \ba$ as in (\ref{eq;notation1}) with $\bx\in L^k$ one has 
\begin{align}\label{eq;effective goal}
\left |I_f(\bx; \ba; \Psi )-\int_U f\dd\mu_U \cdot\prod_{i=1}^k \int_X \psi _i\dd\mu_X\right|
\le M e^{-\delta \lfloor \ba \rfloor} . 
\end{align}
\end{thm}

We only get  nontrivial estimates in Theorem \ref{thm;effective}
when     
all the $a_i \ (1\le i\le k)$ belong
to 
\[
\{a\in A_U^+: H_b^+ \le H_{a}^+ \}
\]
which is a nonempty open cone. 
Theorem \ref{thm;effective} is proved by an inductive argument in  Theorem \ref{prop;explicit} which is also  used to prove the  multiple ergodic theorem in the appendix and the  central limit theorem in \cite{shi}. 
In applications it is useful to know the explicit dependence of $M$ on $\Psi$ and $k$. We will make this calculation during the proof in \S \ref{sec;effective}. 

It follows from the structure theory of parabolic subgroups of $H$   that $H_b^+$ is the unipotent radical of a  parabolic subgroup containing $P$
and there are only finitely many of them. 
Although given a parabolic subgroup $P_0\ge P$ there might be many choices of $\{b_t\}$ such that 
the unipotent radical of $P_0$ is equal to $H_b^+$, different choices will only affect the exponent 
$\delta$ in Theorem \ref{thm;effective}.  

Theorem \ref{thm;effective} generalizes Kleinbock-Margulis \cite[Theorem 1.3]{km12} and
Kleinbock-Shi-Weiss \cite[Theorem 1.2]{ksw}. 
Here we briefly explain the ideas involved in the proof which is due  to Kleinbock-Margulis \cite{km12}.
We chop the $U$-slice into pieces of $H_{b}^+$-slices and 
for most  $H_b^+$-slices which are not far away in the cusp we use the thickening method and spectral gap 
to get a rate. 
Using the assumption that $a_i\in A_U^+$ we can apply the quantitative nonescape of mass in Theorem \ref{thm;general}  to control effectively   the total mass of  pieces of $H_b^+$-slices possibly far away 
in the cusp.

In a recent paper of Dabbs, Kelly and Li \cite{dkl}, it is noticed that when  $k=1$ and $f$ is  the 
characteristic function of  a fundamental domain of a  periodic orbit  $Ux_1$, then 
there are still nontrivial estimates without assuming  $H_b^+\le H_{a_1}^+\cap U$.   
The following theorem generalizes \cite[Theorem 2]{dkl}.
\begin{thm}\label{thm;periodic}
	Let the notation and assumptions be as in Theorem \ref{thm;effective}.
We assume in addition that $U=H_b^+$. 
	 Then there exists 
	$\delta=\delta(H, X, \{b_t \})>0$ with the following properties:
	For any compact subset $L$ of $X$ and any $\psi\in \mathcal S(X)
	$, there exists $M=M(L, \psi)>0$
  such that for any $x\in L$ with $Ux$ periodic  and $a\in A_U^+$ one has
\begin{align}
\left |\int_F \psi(ahx)\dd\mu_U(h)-\mu_U(F)  \int_X \psi \dd\mu_X\right|
\le M e^{-\delta \lfloor a \rfloor'},
\end{align}
where $F\subset U$ is a fundamental domain of the periodic orbit $Ux$ and 
\[
\lfloor a \rfloor'=\sup\{t\ge 0: ab_{-t}\in A_U^+ \}. 
\]
\end{thm}
It is not known whether there are effective estimates  or not
in the setting of 
Theorem \ref{thm;periodic} for $k\ge 2$. 
But the current method has obvious obstructions which we will explain in the proof.

As an application of effective double equidistribution (the case where $k=2$) we obtain the following result which gives an error rate of  pointwise 
equidistribution.

\begin{cor}\label{cor;effective}
Let the notation and assumptions  be as in Theorem \ref{thm;effective}.  
Let   $ s \in \ap^+$ (see (\ref{eq;cone}))   such that $H_b^+\le U\cap H_{\exp s}^+$. 
Then given 
$\psi \in  \mathcal S(X)
, x\in X$ and   $\varepsilon>0$,   one has 
 \begin{align}\label{eq;pointwise}
\frac{1}{T}\int_0^T \psi (\exp (t\,  s )hx )\dd t=\int_X\psi \dd \mu_X+o(T^{-\frac{1}{2}}\log^{\frac{3}{2}+\varepsilon} T)
\end{align}
 for  $\mu_U$ almost every $h\in U$. 
\end{cor}
The notation  $f_2(T)=o(f_1(T))$ means $\lim_{T\to \infty}\frac{f_2(T)}{f_1(T)}=0$. 
This corollary generalizes \cite[Theorem 1.1]{ksw} and strengthens certain ineffective results 
of \cite[Theorem 1.1]{s}.

\textbf{Acknowledgements:}
I would like to thank Jinpeng An,  Yves Benoist, Dmitry Kleinbock, Jean-Francois Quint and  Barak Weiss  for the discussions related to this work.   Part of this work is done during the author visiting 
the Mathematical Sciences Research Institute (MSRI) in Spring 2015 and Shanghai Center for Mathematical Sciences
in July 2015. We would like to thank both institutes
 for their hospitality.

\section{Properties   of Expanding cone}\label{sec;cone}
In this section we discuss several properties of expanding cones. We prove  Theorem \ref{thm;open}
in \S \ref{sec;explicit}. After that in \S \ref{sec;property}  we prove some properties of the expanding cone including the characterization of 
a sequence drifting away from walls, which will lead to the proof of Theorem \ref{thm;translate}.
In \S\ref{sec;cone verify} we show that the  cone  in Example \ref{ex;main} is  the expanding cone. In the last section we
show that the intersection of the expanding cone with a maximal  $\Q$-split torus is the same as  some  cones considered in \cite{ms}.

 Let the notation be as in \S \ref{sec;a cone}. 
In particular, $H$ is a connected semisimple Lie group without compact factors, 
$U$ is the unipotent radical of an absolutely proper parabolic subgroup $P$ of $H$,  $A$ is a maximal connected
$\Ad$-diagonalizable subgroup in $P$. We introduce more notation in \S \ref{sec;explicit}, which will 
also be used in later sections. 

\subsection{Explicit description}\label{sec;explicit}

Recall that $\Phi=\Phi( \mathfrak h,\mathfrak a)$ is the relative root system of the Lie algebra $\mathfrak h$. 
For every  $\alpha\in\Phi$
 let $\mathfrak h _\alpha$ be the root space of $\alpha$, i.e.
\[
\mathfrak h _\alpha=\{v\in\mathfrak h : [s, v]=\alpha(s)v \mbox{ for all } s\in \mathfrak a\}.
\]
Let $\Phi^+$ and $\Pi$
be the sets of positive and simple roots,
  respectively,  so that  $P$ is a standard  
parabolic subgroup. 
Let   $\mathfrak p$ be the Lie algebra of $P$ and let
$\mathfrak p =\mathfrak l_{\mathfrak p}\oplus \mathfrak u$
be the Levi decomposition such that $\mathfrak l_{\mathfrak p} $ is normalized by $\mathfrak a$. 
There is a subset $I$ of $\Pi $ such that 
\begin{align}\label{eq;levi}
\mathfrak l_{\mathfrak p}=Z_{\mathfrak h }(\mathfrak a)\oplus
\sum_{\alpha\in \Phi_I} \mathfrak h _\alpha \quad {and}\quad  
\mathfrak u  = \sum_{\alpha\in \Phi(\mathfrak u)} \mathfrak h _\alpha, 
\end{align}
where  
$\Phi_I=\Phi\cap \mathrm{span}_\R\, I$ (here and hereafter  $\mathrm{span}_{\mathbb F}$  denotes the linear span over the  field $\mathbb F$ of a  set).  
Let $\mathfrak u^-$ be the 
unipotent radical of the opposite parabolic subalgebra 
$\mathfrak p^-$   of $\mathfrak p$.
 Since we assume $U$ is the unipotent radical of an absolutely proper parabolic subgroup $P$, the Lie algebra 
 generated by $\mathfrak u$ and $\mathfrak u^-$ is $\mathfrak h$. 
 Also, the $\R$-linear span of $
  \Phi(\mathfrak u)
  $ is equal to $\mathfrak a^*$.  
Recall that  
for each $\alpha\in \mathfrak a^*$ we associate $s_\alpha\in \mathfrak a$ via the Killing form
$\mathbf B(\cdot, \cdot)$. 
Let  $\langle \cdot\, ,\cdot\rangle$ be the inner product on $\mathfrak a^*$ defined by 
$\langle \alpha, \beta \rangle=\mathbf B(s_{\alpha}, s_{\beta})$.

 Epimorphic subgroups  were studied   earlier by    Bien and Borel \cite{bb92}. 
 A closed subgroup $S$ of $H$ is said to be epimorphic if for every representation
 $(\rho, V) $ of $H$ one has $V^S=V^H$. 
 In this section all the representations of Lie groups and Lie algebras are on real vector spaces. 
 A Lie group representation $(\rho, V)$ induces a 
 Lie algebra representation on $V$ and we still  use $\rho$ to denote the corresponding Lie algebra 
 representation
 $\mathfrak h \to \mathfrak {gl}(V)$. Let $\mathfrak s$ be the Lie algebra of $S$. 
 It can be checked directly that if $S$ is connected then 
 $V^S$ is equal to the set of $\mathfrak s$ annihilated elements 
  \begin{align*}
 V^{\mathfrak s}:=\{v\in V: \rho(s)v=0 \mbox{ for every  } s\in \mathfrak s\}. 
 \end{align*}
We say $\mathfrak s$ is an epimorphic  subalgebra of $\mathfrak h $ if   for every representation
  $(\rho, V)$ of $\mathfrak h $ one has $V^{\mathfrak h}=V^{\mathfrak s }$.
  It is easy to see that a connected and closed subgroup $S$ of $H$ is epimorphic
  if and only if its Lie algebra $\mathfrak s$ is  epimorphic.

    \begin{lem}\label{lem;epimorphic}
$\mathfrak a\oplus\mathfrak u$ is an epimorphic subalgebra of $\mathfrak h $. 
\end{lem}
\begin{proof}
Let $(\rho, V)$ be a representation of $\mathfrak h $ and $v\in V^{\mathfrak a\oplus \mathfrak u}$. 
Let $\alpha\in \Phi(\mathfrak u)$ and  $u^-\in   \mathfrak h _{-\alpha}$ with $u^-\neq 0$. 
By  \cite[Proposition 6.52]{knapp} 
there exists $u\in \mathfrak u$ and $s\in \mathfrak a$ such that 
$(s, u , u^-)$ is an $\mathfrak {sl}_2$ triple, i.e.~$[s, u]=2u, [s, u^-]=2u^-, [u,u^-]=s$. Since $v$ is annihilated by $s$ and $u$, the representation
theory of $\mathfrak {sl}_2$ (see e.g.~\cite[\S I.9]{knapp}) implies that $v$ is also annihilated by $u^-$. 
Therefore, $v$ is annihilated by $\mathfrak u^-$. 
Since $\{ s\in \mathfrak h: \rho(s)v=0\}$ is a subalgebra and 
the Lie algebra generated by $\mathfrak u$ and $\mathfrak u^-$ is $\mathfrak h $, 
one has $v$ is annihilated by $\mathfrak h $.
\end{proof}

 For $\beta\in \mathfrak a^*$ the subspace of weight $\beta$ in 
 a
 representation $(\rho, V)$ of $\mathfrak h $   
  is
\[
V_\beta:=\{v\in V: \rho(s)v=\beta(s)v \mbox{ for every } s\in \mathfrak a\}. 
\]
 The set of weights  $\Delta  $   consists of $\beta\in \mathfrak a^*$ such
 that $V_\beta \neq \{0\}$ for some representation $V$ of $\mathfrak h$. 
Note that  $\Delta $ is  a lattice of  $\mathfrak a^*$ and 
\begin{align}\label{eq;weight space}
\Delta=\{\beta\in \mathfrak a^*: 2\langle \beta, \alpha \rangle/\langle \alpha,\alpha\rangle \in \Z \mbox{ for every }\alpha \in \Phi\},
\end{align}
see  \cite[Proposition 4.3]{bt}.
By 
  (\ref{eq;weight space}) and standard arguments of Lie algebras (e.g.~the proof of  \cite[Proposition 4.62]{knapp}) one has 
\[
\Delta=\{\beta\in \mathfrak a^*: 2\langle \beta, \alpha \rangle/\langle \alpha,\alpha\rangle\in \Z\mbox{ for every }\alpha \in \Pi'\}, 
\]
where $\Pi'=\{\alpha: \alpha\in \Pi  \mbox{ but }2\alpha\not\in \Phi\}\cup \{2\alpha: \alpha\in \Pi  \mbox{ and }2\alpha\in \Phi\}$.

Let $\mathcal S$  be a nonempty subset of a real vector space $V$. 
The dual cone of $\mathcal S$ is defined to be
\[
{\mathcal S}^\bigstar=\{\varphi\in V^*: \varphi( v) \ge 0 \mbox{ for every } v \in \mathcal S\}. 
\] 
We will consider the  dual cone either for $V=\mathfrak a$ or $V=\mathfrak a^*$ where in the latter case 
we identify $(\mathfrak a^*)^*$ with $\mathfrak a$ in the natural way.  
It follows from the definition of $\mathfrak a_{\mathfrak u}^+$ that 
\[
\overline{\mathfrak a_{\mathfrak u}^+}=\Big \{\sum_{\alpha\in \Phi(\mathfrak u)} t_\alpha s_\alpha: t_\alpha\ge 0\quad \mbox{for all } \alpha\Big \}.
\]
According to the definition of $s_\alpha$ and ${\mathcal S}^\bigstar$ one has 
\begin{align}\label{eq;dual cone}
(\overline{\mathfrak a_{\mathfrak u}^+})^\bigstar=\mathcal C:=\{\beta\in \mathfrak a^*: 
\langle \beta, \alpha \rangle\ge  0 \mbox{ for all } \alpha\in \Phi(\mathfrak u)\}. 
\end{align}

\begin{lem}\label{lem;interior}
The interior of  $\overline{\mathfrak a_{\mathfrak u}^+}$ is $\ap^+$. 
\end{lem}

\begin{proof}
It is clear that $\ap^+$ is a subset of the interior of $\overline{\mathfrak a_{\mathfrak u}^+}$. 
On the other hand suppose   $s$ is in the  interior of $\overline{\mathfrak a_{\mathfrak u}^+}$.
For every 
$\alpha\in \Phi(\mathfrak u)$ we choose a positive real number 
$\varepsilon_\alpha$ sufficiently small so  that 
\[
s-\sum_{\alpha\in \Phi(\mathfrak u) }\varepsilon_\alpha s_\alpha\in \overline{\mathfrak a_{\mathfrak u}^+}.
\]
Therefore, we can write  
\[
s=\sum_{\alpha\in \Phi(\mathfrak u) }(t_\alpha+\varepsilon_\alpha) s_\alpha
\]
where $t_\alpha\ge 0$. 
Hence $s\in \ap^+$. 
\end{proof}

\begin{lem}\label{lem;new interior}
For every $\beta\in \mathcal C\setminus\{0\}$ one has
 $\mathfrak a_{\mathfrak u}^+
\subset \{s\in \mathfrak a: \beta(s)>0\}$.
\end{lem}
\begin{proof}
It follows from the definition of $\mathcal C$ in (\ref{eq;dual cone}) that $\beta(s_\alpha)\ge 0$
for all $\alpha\in \Phi(\mathfrak u)$. Since $\beta\neq 0$ one has $\beta(s_\alpha)>0$ for some $\alpha\in
\Phi( \mathfrak u)$. 
Note that  every element of $\mathfrak a_{\mathfrak u}^+$ is a positive linear combination of $s_\alpha$
 $(\alpha\in \Phi(\mathfrak u))$, so $\beta(s)>0$ for all $s\in \mathfrak a_{\mathfrak u}^+$. 
\end{proof}

\begin{lem}\label{lem;cone reduction}
There exists a finite set $C\subset \mathcal C\cap (\Delta\setminus\{0\}) $ such that 
\begin{align}\label{eq;optimal 1}
\ap^+=\{s\in \mathfrak a:\beta(s)> 0 \mbox{ for all } \beta\in C\}. 
\end{align}
\end{lem}
\begin{proof}
Using the theory of convex  polyhedral cones
(see e.g.~\cite[Chapter 1]{z}),  it can be shown that there exists 
a finite set $C$ of $\mathrm{span}_{\Q}\,\Delta\setminus \{0\}$ such that 
\begin{align}\label{eq;optimal}
\overline {\ap^+}=\{s\in \mathfrak a: \beta(s)\ge 0\mbox{ for all }\beta\in C\}. 
\end{align}
It follows directly from the definition of $\mathcal C$ in (\ref{eq;dual cone}) that  $C\subset \mathcal C$. Moreover, by possibly  replacing 
$C$ by its integral multiples one can  take $C\subset \Delta\setminus \{0\}$.

It follows from  Lemma \ref{lem;interior} that  $\ap^+$ is the  interior of $\overline{\ap^+}$.
So  (\ref{eq;optimal})
implies 
\[
\{s\in \mathfrak a: \beta(s)>0\mbox{ for all }\beta \in C\}\subset \ap^+. 
\]
The reverse inclusion follows from Lemma \ref{lem;new interior}. 
\end{proof}

\begin{lem}\label{lem;cone}
Let $(\rho, V)$ be a nontrivial irreducible  representation of $\mathfrak h $. 
If  $\beta\in \Delta $ and 
$V_\beta\cap V^{\mathfrak u}\neq \{0\}$ then $\beta\in \mathcal C\setminus \{0\} $. 
\end{lem}
\begin{proof}
First we show that $\beta\in \mathcal C$.
Suppose $\alpha\in \Phi(\mathfrak u)$ and $u\in \mathfrak h_\alpha$ is a nonzero element. 
It follows from \cite[Proposition 6.52]{knapp} that  there exists 
$u^-\in 
\mathfrak u^-$ such that
$(s_0=2s_\alpha/\langle \alpha,\alpha\rangle, u, u^-)$ is an $\mathfrak {sl}_2$-triple. 
Note that  \begin{align}\label{eq;cone note}
\beta(s_0)=2\langle \beta,\alpha \rangle/\langle \alpha,\alpha\rangle 
\end{align}
is an eigenvalue of $\rho(s_0)$
and elements of   $V_\beta$ are eigenvectors of it. 
Let $v\in V_\beta\cap V^{\mathfrak u}$ be a nonzero vector.  
Since $v$ is  annihilated by $u$, it is a highest weight vector for some 
 representation of $\mathfrak {sl}_2$. 
Therefore,   $\beta(s_0)\ge 0$.  This and  (\ref{eq;cone note}) imply 
  $\langle \beta,\alpha \rangle\ge 0$. Since $\alpha$ is an arbitrary element of $\Phi(\mathfrak u)$,
  we have $\beta \in \mathcal C$. 

Next we show $\beta\neq 0$. 
Assume the contrary,   that is $V_\beta$ is annihilated by $\mathfrak a\oplus \mathfrak u$. 
By Lemma \ref{lem;epimorphic}, $\mathfrak a\oplus \mathfrak u$ is an epimorphic 
subalgebra of $\mathfrak h $, so $V_\beta\subset V^{\mathfrak h }$. This contradicts the assumption 
that $(\rho, V)$ is a nontrivial irreducible representation of $\mathfrak h $ and hence $\beta\neq 0$. 
\end{proof}

The converse of Lemma \ref{lem;cone} is also true. 
Let $W$ be the Weyl group of the root system $\Phi$.  For every $\alpha \in \Phi $ we let 
$w_\alpha\in W$ be the orthogonal   reflection with respect to $\alpha$. 
Let  $I$ be the subset of $\Pi$  determined by $\mathfrak p$ as in  (\ref{eq;levi}). 
 Let   $W_I$ be the subgroup of $W$ generated by $w_\alpha$
 $(\alpha\in I)$.

\begin{lem}\label{lem;tits}
For every nonzero $\beta \in \mathcal C\cap \Delta $, there is $m\in \N$ and 
 a nontrivial  irreducible representation $(\rho, V)$ 
of $\mathfrak h $ with highest weight $\lambda$  such that  $m \beta\in W_I \lambda$, $V_{m \beta}\neq \{0\}$
  and
 $V_{m\beta} \subset  V^{\mathfrak u}$.
 \end{lem}

\begin{proof}
Recall that $w_\alpha(\alpha')=\alpha'-\frac{2\langle \alpha, \alpha'\rangle }{\langle \alpha, \alpha\rangle}\alpha $, from which 
it is easy to see  that $w_\alpha$ $(\alpha\in I)$ and hence $W_I$
leaves $\Phi(\mathfrak u)$ and $\mathrm{span}_\R\, I $ invariant.
Let $\beta_I$ be the orthogonal projection of $\beta $ to $\mathrm{span}_\R\, I $.
Note that  $\Phi_I$  is a relative root system in $\mathrm{span}_\R\, I$ with Weyl group $W_I$ and 
 the restriction of $\langle, \rangle $ to  $\mathrm{span}_\R\, I$ is an inner product invariant under $W_I$.
  Since Weyl group 
 acts transitively on closed Weyl Chambers, 
 there is $w\in W_I$ such that $\langle w\beta_I, \alpha\rangle\ge 0 $ for every  $\alpha\in \Phi_I^+$
where $\Phi_I^+=\Phi_I\cap \Phi^+$. 

We claim that  $w\beta $ is dominant with respect to  $\Phi^+$. For $\alpha\in I $ one has 
\[
\langle w\beta, \alpha\rangle =\langle w\beta_I , \alpha\rangle\ge 0,
\]
since $\beta-\beta_I$ is orthogonal to $\mathrm{span}_\R\, I$. 
If  $\alpha\in \Pi\setminus I$ then $\alpha\in  \Phi(\mathfrak u)$. In this 
case 
\[
\langle w\beta , \alpha \rangle=\langle \beta , w^{-1}\alpha \rangle\ge 0,
\]
since $w^{-1}$ leaves $\Phi(\mathfrak u)$ invariant and $\beta \in \mathcal C$.

 Let $\mathfrak t$ be a maximal abelian subalgebra containing $\mathfrak  a$ in  
\[
Z_{\mathfrak h}(\mathfrak a):=\{ h\in \mathfrak h: [h, s]=0 \ \forall s\in \mathfrak a\}.
\] 
The Lie algebra $\mathfrak t \otimes \C $ is a Cartan subalgebra of $\mathfrak h\otimes \C $. There is a positive 
system of the root system $\Phi( \mathfrak h\otimes \C,\mathfrak t \otimes \C )$ such that $\Phi^+$ is the 
the restrictions 
of it. Note that we can write $\mathfrak t=\mathfrak c\oplus\mathfrak a$ such that all the weights  of 
$\mathfrak h\otimes \C$ are real on $i\mathfrak c\oplus \mathfrak a$. 
Since we assume   $\beta\in \Delta$,  there is a representation $V'$ of $\mathfrak h$ such that     $w\beta$ 
 a  weight of it. Let  $\chi_1, \ldots, \chi_m$ be the $\mathfrak t\otimes \C$-weights on
 the space $V'_{w\beta}\otimes \C$ counted with  multiplicity. Then 
 $\chi=\sum_{i=1}^m\chi_i$ is  algebraically integral  with respect to $\Phi( \mathfrak h\otimes \C,\mathfrak t \otimes \C )$
 and  $\chi$ vanishes on $\mathfrak c$. 
 Since $\Phi( \mathfrak h\otimes \C,\mathfrak t \otimes \C )^+$ is compatible with $\Phi^+$ and $w\beta$ is dominant
 (with respect to $\Phi^+$) one has
  $\chi$ is dominant.

  It follows from the highest weight 
  theory of complex semisimple Lie algebras that there is an irreducible complex  representation  of $\mathfrak h\otimes \C$
  with highest weight
   $\chi$.
   This representation can be viewed as a real representation of $\mathfrak h$ and we let
 $V$ be an irreducible sub-representation  of it. 
Note  that $m (w\beta)$ is the highest weight of $V$,
so 
$V_{m(w\beta)}$ is nonzero and  contained in $V^{\mathfrak u}$. 
Let $\widetilde H$  be 
the simply connected covering    of $H$. 
There is a  representation of $\widetilde H$ on $V$ inducing the Lie algebra representation of $\mathfrak h$.
It follows from \cite[Theorem 6.57]{knapp} that  there is $h\in \widetilde H$ such that 
$\mathrm{Ad}(h)$ leaves $\mathfrak a$ invariant and the induced action on $\mathfrak a^*$ is exactly $w$. 
This observation and the fact that   $w$ leaves $\Phi(\mathfrak u)$ invariant imply  
$V_{m\beta}$ is nonzero and contained in $V^{\mathfrak u}$.
\end{proof}

 \begin{proof}[Proof of Theorem \ref{thm;open}]
 By  Lemma \ref{lem;cone reduction} there exists 
 a subset $C$ of $\mathcal C\cap (\Delta \setminus\{0\})$ 
 such that (\ref{eq;optimal 1}) holds. 
  Let  $a=\exp s$ for some $s\in \mathfrak a$.

 First, we assume $a\in A_U^+$ and  prove that $s\in \mathfrak a_{\mathfrak u}^+$. 
 By (\ref{eq;optimal 1}) it suffices to show $\beta(s)>0$ for all
  $\beta\in C$. 
 It follows from Lemma \ref{lem;tits} that there is $m\in \N$ and  a nontrivial  irreducible representation $(\rho, V)$ of $\mathfrak h $
 with highest weight $\lambda $
 such that $m\beta\in W_I\lambda$, $V_{m\beta}\neq\{0\}$ and $V_{m\beta}\subset  V^{\mathfrak u}$. 
There exists  $n\in \N$ such that the representation $(\rho^{\otimes n}, V^{\otimes n})$ lifts 
to a representation of the  Lie group  $H$.  It follows that there is a nontrival irreducible subrepresentation $(\til \rho, \til V )$
of $H$  in $V^{\otimes n}$ such that $\til {V}_{mn\beta}\neq 0$ and $\til {V}_{mn\beta}\subset \til{ V}^{ U }$. 
 It follows from the definition of $A_U^+$ in (\ref{eq;expanding cone}) that $\til V^{ U }\subset \til V_a^+$. 
 Therefore $\til V_{mn\beta}\subset \til V_a^+$. 
 Since elements of  $\til V_{mn\beta}$ are  eigenvectors of $\til \rho(a)$ with eigenvalue $e^{mn\beta(s)}$,
 one has $\beta(s)>0$. This completes the proof  that $s\in \ap^+$. 
 
 Conversely, we assume $s\in \ap^+$ and prove that $a\in A_U^+$.
 Let
  $(\rho, V)$ be a nontrivial irreducible representation of
  $H$.
  Suppose $\beta$ is a weight
 in  $V$ with $V_\beta\cap V^{\mathfrak u}\neq\{0\}$.
 It follows from Lemma \ref{lem;cone} that   $\beta\in \mathcal C\setminus \{0\}$.
 According  to Lemma \ref{lem;new interior}  and the assumption 
 $s\in \mathfrak a_{\mathfrak u}^+$,
  one has $\beta(s)>0$. 
  This implies 
 that $V_\beta\subset V_a^+$. Note that $V^{\mathfrak u}=V^{ U } $ and this space is  $A$-invariant.
 Therefore  $V^{ U }\subset V_a^+$. 
By  the definition of $A_U^+$,
 one has  $a\in A_U^+$.
 \end{proof}
 
 \subsection{Some properties}\label{sec;property}
 
 We prove  more properties of  the expanding cone $A^+_U$.  
 Although we may not use all of these facts  in this paper, we think they have  their own interests and 
 will help the reader understand  this concept. 
 
   \begin{lem}\label{lem;epimorphic new}
  For  every $a\in A_U^+$, the group $\{a^m: m\in \Z\} U  $ is epimorphic in $H$.
 \end{lem}
 \begin{proof}
 Let $(\rho, V)$ be a nontrivial irreducible representation of $H$.
 Since $a\in A_U^+$, for any $v\in V^U$ one has $v\in V_a^+$. In particular if $v$ is nonzero, 
 $\rho(a)v\neq v$. So  there
 is no nonzero  invariant vector  for the group $\{a^m: m\in \Z\} U  $. 
 \end{proof}

 Let $d_{\mathfrak a}$ be the  metric on  $\mathfrak a$
 induced by the 
 Killing form.
We use $d_{\mathfrak a}(s, \partial \ap^+)$ to 
denote the distance between  $s$ and the boundary of $\ap^+$.
We say a sequence $ \{s_i\}_{i\in \N}$ of elements of $\ap^+$  drifts away
from walls if $d_{\mathfrak a}(s_i, \partial \ap^+)\to \infty$ as $i\to\infty$. 
By  Theorem~\ref{thm;open}, $\ap^+$ can be identified  with $A_U^+$ via the  exponential map. So  $\{s_i\}_{i\in \N}$ drifts away from walls if and only if $\{\exp s_i\}_{i\in \N}$ does in the sense defined in  \S\ref{sec;intro translate}. 
 \begin{lem}\label{lem;drift}
 A sequence  $\{s_i\}_{i\in \N}$  of elements of $ \ap^+$ drifts away from walls if and only if
\begin{align}\label{eq;drift}
\beta(s_i)\to \infty \quad \mbox{as } i\to \infty
\end{align}
 for every $\beta\in \mathcal C\setminus\{0\}$.  
 \end{lem}
 \begin{proof}
 Let $\beta\in \mathcal C\setminus\{0\}$. 
By Lemma \ref{lem;new interior}, the hyperplane 
$\mathcal H_\beta:=\{s\in \mathfrak a: \beta (s)=0\}$ does not intersect 
$\mathfrak a_{\mathfrak u}^+$. 
 Therefore $d_{\mathfrak a}(s_i,\mathcal  H_\beta)\ge d_{\mathfrak a}(s_i, \partial \ap^+)$. 
So if $\{s_i\}_{i\in \N}$ drifts away from walls then 
(\ref{eq;drift}) holds.
 
 On the other hand  by  Lemma \ref{lem;cone reduction} there is a finite subset $C$
 of $\mathcal C\setminus \{0\}$ such that $C^\bigstar=\overline{\ap^+}$. 
 In particular 
 \[
 \partial \ap^+\subset \bigcup_{\beta\in C} \mathcal H_\beta.
 \]
If (\ref{eq;drift}) holds for every $\beta\in \mathcal C$,  one has 
 \[
 d_{\mathfrak a}(s_i, \partial\ap^+)\ge  d_{\mathfrak a}(s_i,  \bigcup_{\beta\in C}\mathcal  H_\beta)\to \infty \quad \mbox{as}
 \quad i\to \infty. 
 \]
  \end{proof}
  
  \begin{cor}\label{cor;drift away}
  Suppose a sequence  $\{s_i\}_{i\in \N}$  of elements of $\mathfrak a_{\mathfrak u}^+$ drifts aways 
  from walls. Let $(\rho, V)$ be a representation of  $\mathfrak h$ without nonzero $\mathfrak h$-annihilated 
  vectors. Let $\beta$ be a weight on $V^{\mathfrak u}$ with respect to $\mathfrak a$. Then 
  \begin{align}\label{eq;drift away 1}
  \beta(s_i)\to \infty \quad \mbox{as } i\to \infty. 
  \end{align}
  
  \end{cor}
  \begin{proof}
 Since  $V$ has no nonzero $\mathfrak h$-annihilated 
  vectors, it is a direct sum of nontrivial irreducible subspaces. Therefore, Lemma \ref{lem;cone}
  implies $\beta\in \mathcal C\setminus \{0\}$. So (\ref{eq;drift away 1}) follows from Lemma \ref{lem;drift}. 
  \end{proof}

  \begin{proof}[{Proof of Theorem \ref{thm;translate}}]
According to 
 \cite[Theorem 1.4]{sw96} one has  (\ref{eq;translate}) holds 
 if $\{a_i\}_{i\in \N}$ belongs to a proper closed cone of $A_U^+$. 
 But actually if one goes through  the proof of  \cite[Theorem 1.4]{sw96}, one can see that what really needed is \cite[(3.13)]{sw96}. In view of Corollary \ref{cor;drift away}, the assumption 
 $\{a_i\}_{i\in \N}$ drifts away from walls implies \cite[(3.13)]{sw96}, from which   (\ref{eq;translate}) follows. 
 
 \end{proof}

  \subsection{Expanding cone and Example \ref{ex;main}}
  \label{sec;cone verify}
  The aim of this section is to show that in the setting of Example \ref{ex;main}  the expanding 
  cone is given by (\ref{eq;cone example}). 
 We remark here that it is proved in Kleinbock-Weiss \cite[Lemma 2.3]{kw08}
  that the right hand side of (\ref{eq;cone example})  is contained in $A_U^+$.

  We identify $\mathfrak h$ with the space of trace zero matrices of rank $m+n$. Then $\mathfrak a$ is the 
  space of diagonal matrices in $\mathfrak h$, i.e. 
  \[
  \mathfrak a=\{s=\mathrm{diag}(t_1, \ldots, t_{m+n}): \sum t_i=0 \}.
  \]
  For 
  $ 1\le i\le m+n$
  we let $e_i\in (\R^{m+n})^*$ be the element which maps every vector to its $i$-th component. 
  Then the space 
   $
  \{\sum t_i e_i: \sum t_i=0\}
  $
   can be identified with $\mathfrak a^*$.

We take   \begin{align*}
  \Phi& =\{e_i-e_j: i\neq j\}, \\
  \Phi^+&=\{e_i-e_j: i< j\}, \\
  \Pi& =\{e_i-e_{i+1}: 1\le i< m+n\}.
  \end{align*}
  Using the notation of (\ref{eq;levi}), the group $U$ in Example \ref{ex;main} is the 
  unipotent radical of the parabolic group given  by $I=\{e_m-e_{m-1}\}$.
  For $\alpha=e_i-e_j$, the associated element $s_\alpha\in \mathfrak a$
is  the diagonal matrix with entries zero except at $i$-th diagonal entry which is $\frac{1}{4}$  and $j$-th diagonal  entry which is  $-\frac{1}{4}$.

In view of 
 \[
 \Phi(\mathfrak u ) =\{e_i-e_{m+j}:1\le  i\le m ,1\le j\le n\}
 \]
one has 
 \[
 \mathfrak a_{\mathfrak u}^+=\{\mathrm{diag}(r_1, \ldots, r_m, -t_{1}, \ldots, -t_n):r_i, t_j>0\}.
 \]
 So the  cone in Example \ref{ex;main} is equal to  $\exp \mathfrak a_{\mathfrak u}^+$ which is the
expanding cone of $U$ according to Theorem \ref{thm;open}. 

\subsection{Intersection of expanding cone with a subtorus}\label{sec;intersect}

In this section we show that the cone considered in \cite{ms} for minimal $\Q$-parabolic subgroup 
 is the intersection of 
the expanding cone defined in this paper with the corresponding $\Q$-split torus.

We first set  $H, A, U$ in the framework of \cite{ms}.
We assume that  $H=\mathbf H(\R)^{\circ}$ where $\mathbf H$ is a Zariski connected semisimple  algebraic group defined over $\Q$. For simplicity we assume $\mathbf H$ is $\Q$-simple and simply connected. 
Let $\mathbf A$ be the Zariski closure of $A$ in $\mathbf H$. 
We assume that $\mathbf A$ is defined over $\Q$, $\mathbf A$ contains a maximal
$\Q$-split torus $\mathbf A_0$ and $P=\mathbf{P}(\R)$ where $\mathbf P$ is a minimal parabolic subgroup over $\Q$.

Let $A_0=\mathbf A_0(\R)^\circ$,  $\mathfrak a_0=\mathrm {Lie}(A_0)$ and 
$r: \mathfrak a^*\to \mathfrak a_0^*$ be the restriction map. 
Recall  that (see \cite[\S 21.8]{borel})
\[
\Phi_0=\{r(\alpha): \alpha\in \Phi, r(\alpha)\neq 0\}
\]
 is a relative root system and 
$\Phi^+_0=\Phi_0\cap r(\Phi^+)$ is a positive system with the 
 set of simple roots $
\Pi_0=\Phi_0\cap r(\Pi) 
$.

The Weyl group $W$ of $\Phi$ acts orthogonally on $\mathfrak a^*$ and it induces an 
orthogonal action on $\mathfrak a$ via $w(s_\alpha)=s_{w\alpha}$  via the Killing form.
 Let $W_0$ be the restriction of 
$
\{w\in W: w\mathfrak a_0=\mathfrak a_0\}
$
to $\mathfrak a_0$. 
The restriction of the Killing form $\mathbf{B}$ on $\mathfrak a_0$ is an inner product invariant under $W_0$. 
For every $\alpha_0\in \mathfrak a_0^*$ we let $s_{\alpha_0}\in \mathfrak a_0$ be the element determined by 
$\mathbf{B}(s_{\alpha_0}, s)=\alpha_0(s)$ for all $s\in \mathfrak a_0$. 
 Let $W_0$ acts on $\mathfrak a_0^*$ by 
$s_{w(\alpha_0)}=w s_{\alpha_0}$ for $\alpha_0\in \mathfrak a_0^*$ and $w\in W_0$. Then $W_0$ is the Weyl group of $\Phi_0$, see 
\cite[Corollary 21.4]{borel}.
Let 
\[
\mathfrak a_{0, \mathfrak u}^+=
\Big\{\sum _{\alpha_0\in \Phi_0^+}t_{\alpha_0}s_{\alpha_0}: t_{\alpha_0}>0\quad \mbox{for all }\alpha_0\Big   \}.
\]
\begin{prop}\label{prop;intersect}
$\mathfrak a_{\mathfrak u}^+\cap \mathfrak a_0=\mathfrak a_{0, \mathfrak u}^+$.
Moreover a sequence $\{s_i\}_{i\in \N}\subset \mathfrak a_{0, \mathfrak u}^+$ drifts away from 
walls of  $\mathfrak a_{0, \mathfrak u}^+$ in $\mathfrak a_0$ if and only if it drifts away from walls of 
$\mathfrak a_{\mathfrak u}^+$ in $\mathfrak a$. 
\end{prop}
\begin{proof}
	 Let $\mathbf A_1$ be the maximal $\Q$-anisotropic torus of  $\mathbf A$ and let 
	 $\mathfrak a_1$ be the Lie algebra of $A_1:=\mathbf A_1(\R)^\circ$. 
	 Then $A=A_0A_1$ is  a direct product and $\mathfrak a=\mathfrak a_0\oplus \mathfrak a_1$ is an orthogonal 
	 decomposition with respect to the Killing form, see 
	 \cite[Proposition 8.15]{borel}. For all $\alpha\in \Phi$ the element $s_{r(\alpha)}$ is the orthogonal projection of $s_\alpha$ to
	 $\mathfrak a_0$. 
	 Let $\lambda_0\in \mathfrak a^*_0$ be a nonzero  dominant  weight with respect to $\Phi_0^+$ and
	 let  $\lambda\in \mathfrak a^*$ be the extension 
	 of $\lambda_0$ by taking $\lambda(\mathfrak a_1)=\{0\}$.  Then $\lambda$ is  dominant 
	 with respect to $\Phi^+$. In particular $\lambda\in\mathcal C$ which is defined 
	 in (\ref{eq;dual cone}). Therefore for all $s\in \mathfrak a_{\mathfrak u}^+\cap \mathfrak a_0 $
	 one has 
	 \[
	 \lambda_0(s)=\lambda(s)> 0
	 \]
	 by Lemma \ref{lem;new interior}.
	 Hence 
	 $\mathfrak a_{\mathfrak u}^+\cap \mathfrak a_0 \subset \mathfrak a_{0, \mathfrak u}^+$.

Let $k$ be a finite extension of $\Q$  in $\R$ such that  $\mathbf A$  is
 $k$-split. 
 The group $\mathrm{Gal}(k/\Q)$ acts naturally  on $\Phi$ and hence on $V=\mathrm{span}_\Q\, \Phi$. 
Note that 
 $r(\sigma(\alpha))=r(\alpha)$ for all  $\alpha\in V$ and $\sigma \in 
\mathrm{Gal}(k/\Q) $.
 Also,  $\alpha\in V$ vanishes on $\mathfrak a_1$ if and only if $\sigma(\alpha)=\alpha$
 for all $\sigma\in \mathrm{ Gal}\, (k/\Q)$.
  It follows from the above two observations  that for   $\alpha_0\in \Phi_0^+$
\begin{align*}
\sum_{\alpha\in \Phi, r(\alpha)=\alpha_0} \alpha=m\alpha_0',
\end{align*}
where $m$ is a positive integer and $\alpha_0'\in \mathfrak a^*$ is the extension of $\alpha_0$
by taking  $\alpha_0'(\mathfrak a_1)=\{0\}$. 
Therefore 
\begin{align}\label{eq;sanya}
s_{\alpha_0}= s_{\alpha_0'}=\frac{1}{m}\sum_{\alpha\in \Phi, r(\alpha)=\alpha_0}s_\alpha.
\end{align}
Since $\alpha\in\Phi(\mathfrak u)$ if and only if $r(\alpha)\in \Phi_0^+$, 
(\ref{eq;sanya}) implies  $\mathfrak a_{0, \mathfrak u}^+\subset \mathfrak a_{\mathfrak u}^+\cap \mathfrak a_0$. 

For the  second assertion, we recall that the cone $\mathfrak a_{\mathfrak u}^+$ is a  polyhedral cone with $0$ as a vertex. 
Therefore the boundary of 
 $\mathfrak a_{0,\mathfrak u}^+$ in $\mathfrak a_0$ is exactly the intersection of $\partial \mathfrak a_{\mathfrak u}^+$ with 
 $\mathfrak a_0$.  The second assertion follows from the observation that a sequence of points in the polyhedral cone drift away
 from walls if and and only if the distance of the sequence to each hyperplane of the boundary goes to infinity (cf. Lemma \ref{lem;drift}). 
\end{proof}

Suppose  $\{ s_i\}\subset \mathfrak a_{0, \mathfrak u}^+$  drifts away from  walls in
$\mathfrak a_0$. Let  $a_i=\exp s_i $ and  $X=H/{\mathbf H}(\Z)\cap H$. 
Let  $\nu$ be  a probability measure on $X$ 
supported a $U$ orbit and absolutely 
continuous with respect to the volume measure of this orbit. 
By Proposition \ref{prop;intersect} and   Theorem \ref{thm;translate}, one has 
$a_i \nu$ converges to the probability Haar measure on $X$ as $i\to \infty$. 
The same equidistribution result holds for any probability measure 
on $X$ whose disintegration into measures on $U$ orbits   have the same property as 
$\nu$. 
In particular, this  applies to homogeneous  measures considered in \cite{ghsy} and gives a proof of a special case of \cite[Theorem 1(\rmnum{2})]{ghsy}.
 

\section{Quantitative nonescape of mass
}
\label{sec;U slice}

In this section we  prove Theorem \ref{thm;general}.  Recall that   $H$ is  
a subgroup of a connected semisimple Lie group $G$, $\Gamma$ is a lattice of $G$ and $X=G/\Gamma$. 

Let $\mathfrak u$ be the Lie algebra of $U$ and we endow $\mathfrak u$ with an 
inner product structure  such that the corresponding 
Lebesgue measure $\mu_{\mathfrak u}$ is mapped to the Haar measure $\mu_U$ through  the 
exponential map, see  \cite[Theorem 1.2.10 (a)]{cg}. 
We use $B_r^{\mathfrak u}$ to denote the open  ball of radius $r$ centered at $0\in \mathfrak u$ with respect to 
the metric induced from this inner product. 
 For $r>0$ there exists $\ell\ge 1$ such that  
\begin{align}\label{eq;lie algebra}
 \exp B_{ r/\ell}^{\mathfrak u}\subset B_r^U\subset \exp B_{\ell r}^{\mathfrak u}. 
\end{align}
Since we allow the implied  constant in (\ref{eq;better}) depends on $r$ we can work with $(\mathfrak u, \mu_{\mathfrak u})$
instead of $(U , \mu_U)$.

\subsection{Arithmetic quotient}\label{sec;arithmetic}

\begin{lem}\label{lem;norm}
Let $V$ be a real vector space with norm $\|\cdot\|$ and 
let $\rho: H\to \GL(V)$ be  a continuous homomorphism. For every 
 $r>0$ there exists $t >0$ (depending on $\rho, r $ and $\|\cdot\|$) such that for any $ a\in \overline{A_U^+}$ and any $v\in
 V\setminus\{0\}$ one has
 \begin{align}\label{eq;norm}
 \sup_{h\in B_r^U} \frac{ \|\rho(ah)v\|}{\|v\|}\ge t.
 \end{align}
\end{lem}

\begin{proof}
We decompose $V$ into the direct sum of   $H$-invariant subspaces  $V_1\oplus V_2$ where 
$H$ acts trivially on $V_1$ and $H$ does not  leave  any nonzero vector invariant  in $V_2$. 
Since different   norms on $V$ are equivalent,   we assume without loss of generality that
$\|\cdot\|$ is a Euclidean norm
induced from an inner product 
 with the following properties:
\begin{itemize}
	\item  $V_1$ and $V_2$ 	are orthogonal;
	\item  there exists an   orthonormal basis of $V_2$ under which $\rho(A)$ is diagonal.
\end{itemize}
Let $\pi: V\to V_2^U$ be the orthogonal projection.
Since $V_2^U$ is $A$-invariant, one has  
$\rho(a)\pi=\pi\rho(a)$  for all  $a\in A$.
 Since for  $a\in \overline {A_U^+}$ the  
  eigenvalues of $\rho(a)$ on $V^U_2$ are all greater than or equal to $1$,
one has 
  \begin{align}\label{eq;norm 1}
  \|\rho(a)v\|\ge \|v\| \quad \forall a\in \overline {A_U^+},v\in V_2^U.
  \end{align}

 We first  show  that
\begin{align}\label{eq;norm 2}
t_1:=\inf_{v\in V_2, \|v\|=1}\sup_{h\in B_r^U} \|\pi(\rho(h)v)\|
\end{align}
  is positive. 
  Let $v\in V_2 $ be a nonzero vector. We
 claim  that 
there exists $h\in B_r^U$ such that $\pi(\rho(h)v)\neq 0$.
Suppose the contrary, then the subspace $V_2'$ generated by $\rho(U)v$ is orthogonal 
to $V_2^U$. 
By Engel's  theorem for the  nilpotent Lie algebra $\mathfrak u$ acting  on the space $V_2'$, there exists a nonzero $U$-invariant 
vector 
$v'\in V_2'$. Since  $v'\in V_2^U$, it is not orthogonal to $V_2^U$. This contradiction 
proves the claim. 
Therefore for every 
$v\in V_2$ there exists $t_v>0$ and an open neighborhood $N_v$ of $v$  in $V_2$
 such that  for every $v'\in N_v$
\[
\sup_{h\in B_r^U} \|\pi(\rho(h)v')\|\ge t_v. 
\]
 Since $\{v\in V_2: \|v\|=1\}$ is compact one has $t_1>0$. 

Let $a\in \overline{A_U^+}$ and   $v\in V$ with $\|v\|=1$.  We write $v=v_1+v_2$ according to the orthogonal decomposition $V=V_1\oplus V_2$
and estimate  
\begin{align*}
\|\rho(ah)v\|^2&=\|v_1\|^2+\|\rho(ah)v_2\|^2  && V_1 \mbox{ is orthogonal to }V_2\\
& \ge \|v_1\|^2+\|\pi(\rho(ah)v_2)\|^2  && \pi \mbox{ is orthogonal projection}\\
& \ge \|v_1\|^2+\|\rho(a)\pi(\rho(h)v_2)\|^2  && \rho(a) \mbox{ and } \pi \mbox{ commutes}\\
& \ge \|v_1\|^2+\|\pi(\rho(h)v_2)\| ^2&& \mbox{by } (\ref{eq;norm 1}). 
\end{align*}
By  above inequality and   (\ref{eq;norm 2}) one has 
\begin{align*}
\sup_{h\in B_r^U} \|\rho(ah)v\| \ge 
\left\{
\begin{array}{ll}
1/{2} & \mbox{if }\|v_1\|\ge 1/2\\
{t_1}/{2} & \mbox{if }\|v_2\|\ge 1/2.
\end{array}
\right.
\end{align*}
Therefore  (\ref{eq;norm}) holds for $t=\min\{{1}/{2}, {t_1}/{2}\}$. 
\end{proof}


\begin{lem}\label{lem;arithmetic}
Let 	$\Gamma_0=SL_n(\Z)\cap G_0$  be a lattice of a closed subgroup 
$G_0$ of $SL_n(\R)$
and let   $\rho: H\to G_0$ be a continuous homomorphism. 
Then 
there exists $\delta>0$ such that 
for any compact subset $L$ of $Y=G_0/\Gamma_0$, for any $r, \varepsilon>0$ and any $a\in 
\overline{A_U^+}, y\in L$
\begin{align*}
\mu_U(\{h\in B_{r}^{ U }: \rho(a h)y\not\in \mathrm{Inj}_{\varepsilon}^Y\}) \ll_{L, r} \varepsilon ^{\delta}.
\end{align*}  
\end{lem}

\begin{proof}

Since different right invariant  Riemannian structures  on $G_0$ will induce equivalent metrics on $G_0$ and $Y$, we
assume without loss of generality that the inclusion map 
$G_0\to SL_n(\R)$ is a Riemannian embedding. 
It follows that 
 \[
 I(g\Gamma_0, G_0/\Gamma_0)\ge I(g SL_n(\Z), SL_n(\R)/ SL_n(\Z))
 \]
 for  all $g\in G_0$.
 So it suffices to prove the lemma in the case  where $L_0=SL_n(\R)$ and 
 $\Gamma_0=SL_n(\Z)$,  which we will assume in the rest of the  proof.

  Recall that the Lebesgue measure on  $\mathfrak u$ is mapped to $\mu_U$
 by the exponential map and for any $r>0$ there exists $\ell\ge 1$ such that  (\ref{eq;lie algebra}) holds. 
So it suffices to show  that there exists $\delta>0$ such that 
for any compact subset $L$ of $Y$,  any  $r, \varepsilon>0$ and any $a\in \overline{A_U^+}$,  $y\in L$
\begin{align*}
\mu_{\mathfrak u}(\{u\in B_{r}^{ \mathfrak u }: \rho( a (\exp u))y\not\in \mathrm{Inj}_\varepsilon^Y\})\ll_{L} \varepsilon ^{\delta}\mu_{\mathfrak u}(B_{r}^{ \mathfrak u }).
\end{align*}

  Let $\|\cdot\|$ be the standard Euclidean norm 
 on $V=\oplus_{i=0}^n \bigwedge ^i\R^n$. 
 For $\varepsilon>0$ we let 
 \[
 K_\varepsilon=\{g\Gamma_0\in Y: \inf_{v\in \Z^n\setminus \{0\}} \|gv\|\ge \varepsilon\}.
 \]
 By \cite[Proposition 3.5]{km12} it suffices to show that 
 there exists $\delta>0$ such that 
for any compact subset $L$ of $Y$,   any $r, \varepsilon>0$ and
any  $a\in 
\overline{A_U^+}$, $y\in L$
\begin{align}\label{eq;arithmetic 1}
\mu_{\mathfrak u}(\{u\in B_{r}^{ \mathfrak u }:\rho( a (\exp u))y\not\in K_{\varepsilon}\}) \ll_{L} \varepsilon ^{\delta}\mu_{\mathfrak u}(B_r^{\mathfrak u}).
\end{align}

 For every nonzero 
 $v\in V$ and $a\in \overline{A_U^+}$ we define  the  map 
 $\psi _{a, v}: \mathfrak u\to \R$ by $\psi _{a, v}(u)=\|\rho(a\exp (u)) v\|$.
 Since $\psi _{a,v}(u)^2$ 
  is a polynomial map 
 with degree at most $k$ where $k$ is determined by $\rho$, 
 \cite[Lemma 3.2]{bkm} implies 
$\psi _{a, v}(u)$ is $(C, \alpha)$-good on $\mathfrak u$ where  $C, \alpha>0$
are positive constants depending only on $\rho$. 
Since $L$ is compact, by Mahler's compactness criterion  there exists $t_0>0$ such that for any $g\Gamma_0\in L$ and  any
nonzero pure tensor  
$v\in \oplus_{i=0}^{n}\bigwedge^i \Z^n$ one has 
\begin{align*}
\|gv\|\ge t_0>0. 
\end{align*}
It follows from 
this observation and Lemma \ref{lem;norm}  that  there exists $t>0$ depending on $L$ such that for any
$0\le i\le n$,  any nonzero pure tensor
$v\in \bigwedge ^i \Z^n$, any 
$a\in \overline{A_U^+}$ and any $g\Gamma_0\in L$ one has 
\[
\sup_{u\in B_r^{\mathfrak u}}\|\rho(a(\exp u))g v\|\ge t^i. 
\]
Therefore (\ref{eq;arithmetic 1}) holds for $\delta=\alpha$ according to  \cite[Theorem 2.2]{k08},  
which completes the proof.

 \end{proof}

\subsection{Quotient of rank one Lie group}\label{sec;rank one}

\begin{lem}\label{lem;rank one}
Let $\Gamma_0$ be a lattice of 
 a connected real  rank one semisimple  Lie group
 $G_0 $. Let $\rho: H\to G_0 $ be a continuous homomorphism. 
Then there  exists $\delta>0$ such that 
for any compact subset $L$ of $Y=G_0 /\Gamma_0$,  any $r, \varepsilon>0$ and  any $a\in 
\overline{A_U^+}$, $y\in L$
\begin{align*}
\mu_U(\{h\in B_{r}^{ U }: \rho(a h)y\not\in \mathrm{Inj}_{\varepsilon}^Y\}) \ll_{L, r} \varepsilon ^{\delta}.
\end{align*}  

\end{lem}

We first recall a compactness criterion of $Y$ due to  Garland and Raghunathan \cite{gr70}. 
Let $\mathfrak g_0$ be the Lie algebra of $G_0$ and let  $V=\bigwedge ^k \mathfrak g_0$
where $k$ is the dimension of the unipotent radical $U_0$ of some  proper parabolic subgroup of $G_0$.
We fix a nonzero element $v_0$ in the one dimensional subspace $\bigwedge^k \mathfrak u_0$ where
$\mathfrak u_0$ is the Lie algebra of $U_0$. 
 The following version 
is taken from Kleinbock-Weiss \cite{kw13}.
\begin{lem}[\cite{kw13} Proposition 3.1]\label{lem;compactness}
There exist $g_1, \ldots, g_m\in  G_0$ such that for 
$\Delta=\{\gamma g_i v_0: \gamma\in 
\Gamma_0, 1\le i\le m\}$ the following holds:
\begin{enumerate}
\item $\Delta $ is a discrete subset of $V$;
\item For any $R\subset G_0$, $\pi(R)\subset Y$ is relatively compact if and only if there exists $r>0$
such that for all $ g\in R$ and $v\in \Delta$ one has $\|gv\|\ge r$;
\item There exists $r_0$ such that for any $g\in G_0$, there is at most one $v\in \Delta$ such that $\|g v\|<r_0$. 
\end{enumerate}
\end{lem}

\begin{lem}\label{lem;rank one 1}
Let the notation be as in Lemma \ref{lem;rank one}. Then there exists 
a compact subset $L_0$ of  $Y$ depending on $L$ such that for any $a\in \overline{A_P^+}$
and $y\in L$ one has $\rho(a B_1^U )y \cap L_0\neq \emptyset$.
\end{lem}

\begin{proof}
Let $ r_0$ and $\Delta$ be as in Lemma \ref{lem;compactness}.  Since $L$ is relatively compact in 
$Y$,   Lemma \ref{lem;compactness}(2) implies that there exists $r>0$ such that for any $g\in G_0$ with 
$g\Gamma_0\in L$ and any $v\in \Delta$ one has 
\begin{align}\label{eq;proof one 1}
\|g v\|\ge r. 
\end{align}
Using (\ref{eq;proof one 1}) and Lemma \ref{lem;norm}, one can find a positive number $t $ such that for 
any $a\in \overline {A_P^+}$, $g\Gamma_0\in L$ and $v\in \Delta$
\begin{align}\label{eq;proof one 2}
\sup_{h\in B_1^U}\|\rho(ah) gv\|\ge t. 
\end{align}
The existence of $L_0$ follows from (\ref{eq;proof one 2}) and 
 (2) of  Lemma \ref{lem;compactness}. 
\end{proof}

\begin{proof}[Proof of Lemma \ref{lem;rank one}]
If $\rho(H)$ consists only the identity element, then there is nothing to prove. 
If $\rho(H)$  is nontrivial,  then it follows from the definition of expanding cone that 
$\rho(A_U^+)$ is the expanding cone of $\rho(U)$ in $\rho(H)$. 
By possibly replacing  $H$ by 
$\rho(H)$
we assume without loss of generality that $H$ is a subgroup of $ G_0$. 
Let $\{a_t: t\in \R\}$ be the one parameter subgroup of $G_0$ such that 
$A_U^+=\{a_t: t> 0\}$.

According to Lemma \ref{lem;rank one 1}
there exits $\varepsilon_0>0$ depending on  $L$ such that 
for any $y\in L$  and  any $t\ge 0$ there exists $h_{y, t}\in B_{1}^{U}$ such that
\[
z_{y, t}:=a_th_{y, t} y\in \mathrm{Inj}_{\varepsilon_0}^Y.
\]
It is easy to see that 
\begin{align*}
\{h\in B_{r}^{ U}: a_t hy\not\in \mathrm{Inj}_{\varepsilon}^Y\}
= \{h\in B_{r}^{ U }: [a_t h(h_{y,t})^{-1}a_{-t}]z_{y,t}\not\in \mathrm{Inj}_{\varepsilon}^Y\}.
\end{align*}
Therefore it suffices to show that
there exists $\delta>0$ such that  for any $\varepsilon_0>0$,   $y\in \mathrm{Inj}_{\varepsilon_0}$,
  $t\ge 0$ and $ r, \varepsilon>0$ 
\begin{align}\label{eq;rank one 2}
\mu_U(\{h\in B_{r}^{ U }: a_t ha_{-t}y\not\in \mathrm{Inj}_{\varepsilon}^Y\}) \ll_{\varepsilon_0, r} \varepsilon ^{\delta}.
\end{align}

In view of the comments at the beginning of \S \ref{sec;U slice},
it suffices to show  that 
there exists $\delta>0$ such that  for any $\varepsilon_0>0$,   $y\in \mathrm{Inj}_{\varepsilon_0}$,
  $t\ge 0$ and $ r, \varepsilon>0$ 
\begin{align}\label{eq;rank one 1}
\mu_{\mathfrak u}(\{u\in B_{r}^{ \mathfrak u }: a_t (\exp u)a_{-t}y\not\in \mathrm{Inj}_{\varepsilon}^Y\}) \ll_{\varepsilon_0, r} \varepsilon ^{\delta}.
\end{align}

According to \cite[Theorem 1.1]{bz}
 there exists 
 $\delta>0$ depending on $Y$ such that 
 for any $u\in \mathfrak u$ with $\|u\|=1$ 
 and $\varepsilon_0, y, t, r, \varepsilon$ as before 
 \begin{align}\label{eq;rank one 3}
|\{-r<s<r: a (\exp su)a^{-1}y\not\in \mathrm{Inj}_{\varepsilon}^Y\}| \le_{\varepsilon_0, Y}  \varepsilon ^{\delta} r.
 \end{align}
 We remark here that  to get (\ref{eq;rank one 3}) one needs to use the 
 explicit calculation of the constant $c_\varepsilon$ in the proof 
 of \cite[Theorem 1.1]{bz}. 
 Using polar coordinates for   the Lebesgue measure on $\mathfrak u$ one gets (\ref{eq;rank one 1})
 from (\ref{eq;rank one 3}). This completes the proof.

\end{proof}

\subsection{General case}

In this section $H, G, \Gamma$ are as in Theorem \ref{thm;general}. 
The following two lemmas allow us to reduce the general case in Theorem \ref{thm;general} to the 
arithmetic case and the rank one case considered   in \S \ref{sec;arithmetic} and \S \ref{sec;rank one}
respectively.

Let $Z_G$ be the center of $G$ and   $G^{\mathrm{cpt}}$ be the maximal connected  compact normal  subgroup of $G$.
\begin{lem}\label{lem;lattice decom}
 There exist Lie groups $G_i\ (0\le i\le k)$ and lattices 
 $\Gamma_i$ in $G_i$, where $G_0=SL_n(\R), \Gamma_0=SL_n(\Z)$ and
  $G_i  \ (1\le i\le k)$ are  real  rank one semisimple  Lie groups, so that the following holds:
  There is  a continuous  homomorphism 
$\rho: G \to  \prod_{i=0}^k G_i$  such that 
$\mathrm{ker}\,(\rho)= Z_G G^{\mathrm{cpt}}$  and  $\rho(\Gamma)\cap \prod_{i=0}^k \Gamma_i$
has finite index in $\rho(\Gamma)$. 
\end{lem}
\begin{proof}
Let $\pi: G\to G/G^{\mathrm{cpt}}$ be the quotient map. Since the kernel of $\pi$ is 
$G^{\mathrm{cpt}}$ and $\pi(\Gamma) $ is a lattice in the group $G/G^{\mathrm{cpt}}$, it 
suffices to prove the lemma in  the case where $G$ has no compact factors.

Suppose $G$ has no compact factors. 
 It follows from Borel density theorem that
$Z_G\Gamma$ is discrete in $G$, see \cite[Corollary 5.17]{r}. So $Z_G\Gamma$ is a lattice of $G$ and hence
$Z_G/(Z_G\cap \Gamma)$  is a finite group. 
 Since $Z_G$ is the kernel  of $\Ad:G\to \GL(\mathfrak g)$, by possibly replacing $G$ by 
 $\Ad(G)$ and $\Gamma$ by $\Ad(\Gamma)$
  we assume without loss of generality that the center of $G$ is trivail. There are closed normal subgroups 
 $G_i'\  (0\le i\le k)$ and lattices $\Gamma_i'$ in $G_i'$ with the following properties:  the multiplication map
$
 \prod_{i=0}^k G_i' \to G
 $ is an isomorphism;
  $G_i '\ (1\le i\le k)$ are  real  rank one semisimple  Lie groups;    $\Gamma_0'$
is an arithmetic lattice of $G_0'$ and $\Gamma$ is commensurable with the image of $\prod_{i=0}^k\Gamma_i'$.
Here $\Gamma_0'$ is an arithmetic lattice means that  there is a continuous  injective homomorphism $\rho_0: G_0'\to SL_n(\R)$ 
such that $\rho_0(\Gamma_0')/\rho_0(\Gamma_0')\cap SL_n(\Z)$ is finite. 
It can be verified easily that for $G_i=G_i', \Gamma_i=\Gamma_i'\ (1\le i\le k)$  the map 
$\rho: G\to \prod_{i=0}^kG_i$ which sends $g_0\cdots g_k$ $(g_i\in G_i')$ to $\rho_0(g_0)\times \prod_{i=1}^k g_i$ satisfies the requirement of the lemma. 
\end{proof}

\begin{lem}\label{lem;reduce}
Let $G_0$ be a Lie group, $\Gamma_0$ be a lattice of $G_0$ and  $Y=G_0/\Gamma_0$. 
Let $\rho: G \to  G_0$ be a  continuous homomorphism such that   
$\mathrm{ker}\,\rho= Z_GG^{\mathrm{cpt}}$  and 
 $\rho(\Gamma)/\rho(\Gamma)\cap \Gamma_0$
is  finite. 
Suppose
there exists $\delta>0$ such that 
for any compact subset $L$ of $Y$, any  $r, \varepsilon>0$ and any
$a\in 
\overline{A_U^+}$,
  $y\in L$
   \begin{align}\label{eq;premain}
\mu_U(\{h\in B_{r}^{ U }: \rho(a h)y\not\in \mathrm{Inj}_{\varepsilon}^Y\}) \ll_{L, r} \varepsilon ^{\delta},
\end{align}  
then 
Theorem \ref{thm;general} holds. 
\end{lem}
\begin{proof}

Let $\ell\ge 1$ such that for all $ t\ge 0, g\in B_t ^ G\cap G^{\mathrm{cpt} }$ and $g_c\in G^{\mathrm{cpt} } $
one has $g_c g g_c^{-1}\in  B_{\ell t}^G  $.
By possibly enlarging $\ell$ we assume that $d_{G_0}(\rho(g_1), \rho(g_2))\le \ell\cdot  d_G(g_1, g_2)$ for all 
$g_1, g_2\in B_1^G$. 
Let $n$ be the cardinality of $\rho(\Gamma)/\rho(\Gamma)\cap \Gamma_0 $. 
We choose $  \varepsilon_0 >0$ sufficiently small  such that $n\varepsilon_0< 1$,
$B_{2n \varepsilon_0 }^G\cap g_c\Gamma g_c^{-1}=\{1_G\}$ for all $g_c\in G^{\mathrm{cpt} }$ and $B_{2 \varepsilon_0 }^G $  has no nontrivial  element with order 
less than or equal to $n$.

We claim that for $\varepsilon \le  \varepsilon_0 $ and $g\in G$
if $I(\rho(g)\Gamma_0, G_0/\Gamma_0)>2n\ell \varepsilon $, then    $I(g\Gamma,  X)>\varepsilon$. 
Theorem \ref{thm;general} will follow from the claim and the assumption of the lemma.
Now we prove the claim. 
Suppose $g_1, g_2\in B_\varepsilon^G, \gamma\in \Gamma$ such that 
$
g_1 g \gamma=g_2 g, 
$ 
then 
\[
\gamma=g^{-1}g_1^{-1}g_2g\quad \mbox{and }\quad
\rho(\gamma)^n=\rho(g^{-1}(g_1^{-1}g_2)^ng)\in \rho(\Gamma)\cap \Gamma_0.
\]
Since $\rho(g_1^{-1}g_2)^n\in B^{G_0}_{2n\ell \varepsilon}$ and  $I(\rho(g)\Gamma_0, G_0/\Gamma_0)>2n\ell \varepsilon$,  one 
has 
\[
g^{-1}(g_1^{-1}g_2)^ng\in \mathrm{ker}\, \rho.
\] 
Since $ \mathrm{ker}\, \rho = Z_G G^{\mathrm{cpt} }$,  there exists $g_c\in  G^{\mathrm{cpt} }$ such that 
$(g_1^{-1}g_2)^n\in \mathrm{ker}\, \rho\cap g_c\Gamma g_c^{-1}$.
Since $B_{2n \varepsilon_0 }^G\cap g_c\Gamma g_c^{-1}=\{1_G\}$ and $\varepsilon \le  \varepsilon_0 $, one has 
 $(g_1^{-1}g_2)^{n}= 1_G$. 
 Moreover, we have   $g_1^{-1}g_2=1_G$, since $B_{2 \varepsilon_0 } ^G$ has no nontrivial elements with order less than 
 or equal to $n$.
\end{proof}

\begin{proof}[Proof of Theorem \ref{thm;general}]
Let $\rho$ and $G_i, \Gamma_i (0\le i\le k)$ be as in Lemma \ref{lem;lattice decom}. 
According to Lemma \ref{lem;reduce} it suffices to verify  (\ref{eq;premain}) for $Y=\prod_{i=0}^k G_i/\Gamma_i$. 
By possibly enlarging $L$ we assume that $L=\prod_{i=0}^k L_i $ where $L_i$ is a compact 
subset of $G_i/\Gamma_i$. 
In this case (\ref{eq;premain}) follows  from Lemmas \ref{lem;arithmetic}
and  \ref{lem;rank one}. 
\end{proof}

\section{Effective equidistribution}\label{sec;mixing}

The aim of this section is to prove  results of \S \ref{sec;intro k mixing}.  
As in \S \ref{sec;intro k mixing}, $\Gamma$ is  lattice of 
 a connected semisimple Lie group $G$, $H$ is a normal subgroup of $G$,    $X=G/\Gamma$,  $\mu_X$ is 
the probability Haar measure on $X$ and the action of   $H$ on $X$ 
has a spectral gap. 

\subsection{Notation and background material}\label{sec;4.1}
Recall that 
$\{b_t\}$ is a one parameter subgroup  of $A$  such that the projection 
of $b=b_1$ to each simple factor of $H$ is not the identity element and
$ H _{b}^+\le U$. 
Since $H$ is a normal subgroup of $G$ and $b\in A\le H$ we  have $G_b^+=H_b^+$. 
To simplify the notation we write $G^+=G^+_b$.
The 
 contracting and centralizing subgroups of $b$ in $G$ are   defined as 
 \begin{align*}
&G^-=G^-_b=\{g\in G: b^n gb^{-n}\to 1_G \quad \mbox{as }n\to \infty\}\le H;\\
&G^c=G^c_b=\{g\in G: bg=gb  \}. 
\end{align*}
Let $\mathfrak g, \mathfrak g^+, \mathfrak g^-, \mathfrak g^c$ be the Lie algebras of 
$G, G^+, G^-, G^c$ respectively.

\begin{lem}\label{lem;structure}
The multiplication map $G^-\times G^c \times G^+ \to G $ is a diffeomorphism  onto an open 
subset with full Haar measure. 
\end{lem}
\begin{proof}
Note that  $\Ad  ( G  ^+), \Ad  ( G  ^-), \Ad (  G    ^c )$ are the horospherical and centralizing subgroups of $\Ad( b )$. Moreover, there exists $a\in A$ such that
$\Ad(a)$ is of class $\mathscr A$ in the sense of \cite[Definition 2.1]{mt94} and
 $\Ad (a)$ has the   same horospherical subgroups 
as $\Ad( b )$.
It follows from \cite[Proposition 2.7]{mt94}  that the Lemma holds for $\Ad   (G)  $ and $\Ad ( b) $. The conclusion for 
$  G $ follows from the   observation that   
$\Ad$ is a covering map and the kernel is contained in $G^c$. 
\end{proof}

To simplify the notation we take $  G  ^{-c}=  G  ^-  G    ^c $.
 In this section all the integrals are with respect to  fixed volume  measures.   If 
 $\mu$ is the measure associated to $Y$,  we may
 use 
 $\dd y$ to denote $\dd \mu(y)$.
Moreover, we assume the fundamental domain of $\Gamma$ in $G$ has measure $1$ and 
for every $\psi \in L^1(  G  )$ one has 
\begin{align}\label{eq;haar}
\int_  G  \psi \dd\mu_  G  =\int_{  G  ^+} \int_  {G    ^c}   \int_{  G  ^-}\psi (g^- g^c  g^+)\Delta( g^c )
\dd  g^- \dd  g^c \dd g^+
\end{align}
where $\Delta(g^c)=|\mathrm{det} (\Ad{( g^c )}|_{\mathfrak g^+})| $. 

Elements of the Lie algebra
$\mathfrak g$ are naturally identified with the right invariant vector fields on $G$.
Therefore an element $h\in \mathfrak g$ defines an operator on $\partial^h:   C^\infty(X)\to  C^\infty(X)$ or $\partial^h:   C^\infty(G)\to  C^\infty(G)$  in a natural way. For  $\alpha=(h_1, \ldots, h_k)\in \mathfrak g^k$,  
we set $\partial ^\alpha=\partial ^{h_1} \cdots \partial ^{h_k}$ and $|\alpha|=k$.    
We fix a basis $\mathcal B$ of $\mathfrak g$ such that the intersections of $\mathcal B$ with $\mathfrak g^-, \mathfrak g  ^c  $ and $\mathfrak g^+$
are the basis of the corresponding  subalgebras.
In what follows we use $\|\cdot\|_\ell$
where $\ell\in \Z_{\ge 0}$ for the usual $(2, \ell)$-Sobolev norm on   $ G , X,  G ^-,  G   ^c ,  G ^+$ defined by  the 
corresponding basis
in $\mathcal B$. 
For example,   suppose $\psi \in C^\infty (X)$, then
\[
\|\psi\|_\ell =\max_{|\alpha|\le \ell} \|\partial^\alpha \psi \|_{L^2(X)},
\]
where the maximum is taken over all the  $k$-tuples $\alpha$ ($0\le k\le \ell $) with alphabets in $\mathcal B$.
   We define
\[
W^{2, \infty}(X)= \{\psi \in C^\infty (X): \|\psi \|_\ell< \infty\  \forall\  \ell \in \Z_{\ge 0}\}.
\]
It is easy to see that
 $\mathcal S(X)\subset W^{2, \infty}(X)$.   Let  $\langle \cdot,\cdot \rangle$ be  the 
inner product of  the Hilbert space $L^2(X)$ and let 
$g\psi$ ($g\in G$ and $\psi\in L^2(X)$) be the function $\psi(g^{-1}x )$ on $X$.

\begin{lem}\label{lem;usual mixing}
There exist $\delta_0, E_0>0$  and $\ell_0\in\N$ depending on $G, \Gamma, \{ b_t : t\in \R\}$ such that for any functions $\varphi , \psi\in W^{2, \infty}(X)$  one has
\[
\left |\langle  b_t \varphi , \psi \rangle -\int _X\varphi\dd \mu_X \int_X\psi\dd \mu_X\right |
\le E_0
\|\varphi \|_{\ell_0} \|\psi\|_{\ell_0} \cdot e^{-\delta_0|t|}. 
\]
\end{lem}
\begin{proof}
Let $Z$ be the center of $G$. The group $Z/(Z\cap \Gamma)$ is finite, see
the proof of Lemma \ref{lem;lattice decom}. 
 Therefore by possibly passing to $ G /(Z\cap \Gamma)$ we assume without loss
of generality that $ G $ has finite center. 
Recall that we assume the projection of 
$b_1$ to each simple factor of $H$ is 
nontrivial. So  the lemma follows from \cite[\S6.2.2]{emv}.
\end{proof}

\subsection{Effective equidistribution for $\{ b_t \}$ translates}
Recall that the metric $d_G$ on $G$ is induced from a right invariant Riemannian structure, and
  for $x\in X$ the map  
$\pi_x:  G \to X$ is defined by   $\pi_x(g)=gx$. 
We first generalize  \cite[Theorem 2.3]{km12}. 
For $\psi \in C^\infty (X)$ let 
\[
\|\psi \|_{\mathrm{Lip}}=\sup_{x, y\in X, x\neq y}\frac{|\psi(x)-\psi(y)|}{{d}_X(x, y)}.
\]

\begin{lem}\label{lem;km 12 general}
Let $\theta\in C_c^\infty(G^+ ),0<r<1$, $m=\mathrm{dim}\, G^{-0}$,  $x\in X$ and  $\delta_0, \ell_0$ be as in Lemma \ref{lem;usual mixing}. 
Assume 
\begin{enumerate}[label=(\roman*)]
\item  $\supp \theta\subset B_r^{  G ^+ }$;
\item $\pi_x$ is injective on $B_{2r}^ G $.
\end{enumerate}
Then there exists a constant $E=E(\ell_0, E_0)$  such that  for any $\psi \in W^{2, \infty}(X)$ with $\int_X \psi\dd \mu_X =0$ and  $\|\psi \|_{\mathrm{Lip}}<\infty$, and  any
$t\ge 0$,  one has 
\begin{align}\label{eq;I x bt}
\left |\int_{G^+}\theta(g)\psi(b_tg x) \dd g \right |\le \new { E \big( \|\psi\|_{\ell_0}+\|\psi\|_{\mathrm{Lip}}\big) }
\left(
r^{-\ell_0-\frac{m}{2} }\|\theta\|_{\ell_0} e^{-t\delta_0}    +    r\int_{G^+ } |\theta(g)|\dd g
\right).
\end{align}
\end{lem}
The proof of this lemma is the same as that of \cite[Theorem 2.3]{km12} and we give details  here
for the completeness.  During the proof we will need the following lemma. 
\begin{lem}\label{lem;sobolev}
Let the notation be as in Lemma \ref{lem;km 12 general}.
There is a nonnegative   function $\til\theta: G^{-0}\to \R$ such that 
 \begin{align*}
 \supp \til\theta\subset B_r^{G^{-0}}, \quad
 \int_{G^{-0}}\til\theta=1, \quad  \mbox{and}\ 
 \|\til\theta\|_{\ell}\ll_\ell r^{-\ell-\frac{m}{2}}\quad \forall \ell\in \Z_{\ge 0}.
 \end{align*}
 Moreover,  the function  $\varphi: X\to \R$ defined by  
 \[
 \varphi(g^-g  ^c gx)=\til\theta(g^-g ^c )\Delta({g  ^c })^{-1}\theta(g)
 \]
 for elements in 
 $B_r^{ G ^{-0}}B_r^{ G ^+}x$ and zero otherwise 
 is  a compactly supported  smooth function on $X$ with 
 \begin{align}\label{eq;sobolev2}
  \|\varphi\|_{\ell}\ll _\ell r^{-\ell-\frac{m}{2}}\|\theta\|_{\ell}\quad \forall \ell\in  \Z_{\ge 0}.
 \end{align}
\end{lem}
Remark:  Besides $\ell$  the implied constants in this lemma also depend on  the space  $ X $,  the choice of the  basis
$\mathcal B$ and  the metric on $ G $, which we do not specify here  since we consider them as fixed with $G$.  

\begin{proof}
This lemma follows from 
\cite[Lemma 2.4.7]{km96} and the  assumptions (\rmnum{1}) (\rmnum{2}) of Lemma \ref{lem;km 12 general}.
\end{proof}

\begin{proof}[Proof of Lemma \ref{lem;km 12 general}]
Let $\til \theta$ and $\varphi$ be as in Lemma \ref{lem;sobolev}. In view of (\ref{eq;haar})
for all $t\ge 0$
\begin{align*}
&\left |\int_{G^+}\theta(g)\psi(b_tg x) \dd g-\langle  \varphi,  b_{-t}\psi  \rangle \right| \\
=&\left| 
\int_{ {G^+ }} \theta(g)\psi ( b_t  g x )\dd g
-\int_G \til\theta (g^-g  ^c )\Delta (g  ^c )^{-1}\theta(g)
\psi ( b_t g^-g  ^c gx) \dd(g^-g  ^c g)
\right |\\
=&\left| 
\int_{ {G^+ }} \theta(g)\big (\psi ( b_t  g x )-\psi ( b_t g^-g  ^c gx)\big)\dd g
\right |\\
\le &
\int_{G^+} |\theta(g)|\cdot
\left| \psi ( b_t  g x )-\psi ( b_t g^-g  ^c b_{-t}\cdot b_t gx)\right |\dd g
\\
\new{\le} & \|\psi \|_{\mathrm{Lip}}\cdot r \cdot \int_{ {G^+ }} |\theta(g)|\dd g.
\end{align*}
On the other hand, by Lemma \ref{lem;usual mixing} and  (\ref{eq;sobolev2})
\begin{align*}
|\langle \varphi, g_{-t} \psi  \rangle|\new{\le E_0} \|\varphi \|_{\ell_0}\|\psi\|_{\ell_0} e^{-t\delta_0} 
 \new {\ll  E_0}  r^{-\ell_0-\frac{m}{2}} e^{-t\delta_0}\|\psi \|_{\ell_0} \|\theta\|_{\ell_0}.
\end{align*}
So (\ref{eq;I x bt}) follows from the triangle inequality and the above two estimates.
\end{proof}

\subsection{Proof of effective equidistribution}\label{sec;effective}

Recall that the followings are fixed: $G, H, \Gamma , A,\allowbreak U , \{b_t: t\in \R\}, b=b_1$. 
As in 
\S \ref{sec;4.1} we  take  $G^+, G^c, G^-$ to be the unstable horospherical, central and stable
horospherical subgroups of $b$ in $G$. We fix $\delta_0 >0$ and $\ell_0\in \N$ so that Lemma \ref{lem;usual mixing} holds. Let $\ell\in \N$ such that  $\ell \ge \ell_0+\mathrm{dim}\, G$. 
We choose  $\delta_1$ with $0<\delta_1\le 1$ such that Theorem \ref{thm;general} holds for $\delta =\delta_1$.
Let 
\begin{align}\label{eq;effect delta}
\delta=\frac{\delta_0\delta_1}{3\ell+\delta_1},
\end{align}
which (in view of the dependence of $\delta_0, \delta_1$ and $\ell$) depends on $ H, X$ and $ \{b_t\}$.

Recall that $k\in \N$ and $f, \ba, \bx, \Psi$ are given in (\ref{eq;notation})
and (\ref{eq;notation1}).
Let 
\begin{align}
\|\Psi\|_{\mathrm{sup}} &= \max \left\{   \sup_{x\in X} |\psi_i(x)| 
: 1\le i\le k
 \right \},\\
\|\Psi\|&=\max \left\{\sup_{x\in X} |\psi_i(x)| ,    \|\psi_i\|_{\mathrm{Lip}} 
 ,  \|\psi_i\|_\ell 
   : 1\le i \le k \right \}. 
\end{align}

Instead of estimating $I_f(\bx ;\ba;\Psi)$ directly, we consider more generally
\begin{align}\label{eq;who}
J_f(\bx; \ba;\Psi)=\int_U f(h)\prod_{i=1}^k \psi_i(a_ih x_i)\dd h. 
\end{align}
It is not hard to see that 
$$
I_f(\bx;\ba;\Psi)=J_f(\bx;\tilde \ba;\Psi)\quad 
\mbox{where} \quad \tilde \ba=(a_1, a_2a_1,\ldots, a_ka_{k-1}\cdots a_{2}a_1).$$ 
We first effectively reduce the estimate of (\ref{eq;who}) for $k\ge 2$ to that of $J_f (\bx ', \ba', \Psi')$ where 
 $$\bx'=(x_1, \dots, x_{k-1}), \ba'=(a_1, \ldots, a_{k-1})
\mbox{ and } \Psi'=(\psi_1, \dots, \psi_{k-1}).$$
We set $J_f (\bx ', \ba', \Psi')=\int_U f(h)\dd h $ if $k=1$ so that 
the following theorem also holds in this case.

\begin{thm}
	\label{prop;explicit}
	Let the notation be as in Theorem \ref{thm;effective} and let    $\delta$ be as in (\ref{eq;effect delta}).
	Then there exists $M_0\ge 1$ depending 
	only
	on $L$ and $f$ such that 
	\begin{align}\label{eq;prop new}
\left 	|J_f(\bx; \ba;  \Psi)-J_f (\bx ', \ba', \Psi')\cdot \int_X \psi_k \dd \mu_X 
\right |\le M_0 k\|\Psi\|_{\mathrm{sup}}^{k-1} \|\Psi\|e^{-\delta \min_{i=0}^{k-1}\lfloor a_k a_i^{-1} \rfloor}
	\end{align}
where $a_0=1_H$. 
\end{thm}

\begin{proof}

Since  (\ref{eq;prop new}) holds for $M_0=2\int_U|f(h)|\dd h$ if 
$\min_{i=0}^{k-1}\lfloor a_k a_i^{-1} \rfloor =0$, 
in  the rest of the proof we assume   that $\min_{i=0}^{k-1}\lfloor a_k a_i^{-1} \rfloor=t>0$. 
Let 
\begin{align}\label{eq;effect r}
r=e^{\frac{-t\delta_0}{3\ell+\delta_1}}<1.
\end{align}
Similar to Lemma \ref{lem;sobolev}, there exists  a smooth function   $\theta:  G^+ \to \R_{\ge 0}$  with 
$\supp \theta\subset B_r^{ G^+}$, $\int_{ G ^+} \theta(g)\dd g=1$ and $\|\theta\|_\ell\ll \, r^{-2\ell}$.
It follows that  
\begin{align}\label{eq;induction I}
J_f(\bx;\ba; \Psi )=\int_{ G^+}\theta(g)\int_Uf(h)\prod_{i=1}^k\psi _i(a_ihx_i)\dd h \dd g.
\end{align}

For every $g\in  G^+$  and $0\le i \le k-1$ we let 
\[
g_i= ( b_{-t}a_k a_i^{-1})^{-1} g\, (  b_{-t}a_k a_i^{-1} )\in G^+. 
\]
The definition of $t$ implies that  $\alpha( \log b_{-t} a_k a_i^{-1})\ge 0$ for every $\alpha\in \Phi (\mathfrak g^+)$  and $0\le i\le {k-1}$.
Therefore, 
\begin{align}\label{eq;d 1gi}
d_G(1_ G , g_i)\le d_{G^+}(1_ G , g_i) \ll d_{G^+}(1_ G , g) <r \quad \forall g\in \mathrm{supp}~ \theta\subset G^+.
\end{align}

Making change of variables in (\ref{eq;induction I}) by $h\to  g_0 h$, one has 
\begin{align}
 J_f(\bx;\ba; \Psi ) \label{eq;I f phi}
=
\int_{ G^+}\theta(g)\int_Uf(g_0h)\psi _k(b_t g  b_{-t}a_k   hx_k)\prod_{i=1}^{k-1}\psi _i(g_i a_ihx_i)\dd h\dd g .
\end{align}
It follows from (\ref{eq;d 1gi}) that for all $g\in \supp \theta$ and $1\le i\le k-1$
\begin{align}
|f(g_0h)-f(h)|&\ll _f \, r \label{eq;indicate},\\
|\psi _i(g_i a_ihx_i)-\psi _i(a_i hx_i)|& 
\new{\le  \|{\psi _i}\|_{\mathrm{Lip}}}\, r  .\label{eq;indicate1}
\end{align}
Replacing  $f(g_0 h)$ by $f(h)$
 and $\psi _i(g_i a_ihx_i)$ by 
 $\psi _i(a_i hx_i)$ for $1\le i\le k-1$
 in the integrand of (\ref{eq;I f phi}), and then changing the order of integration,  one has 
\begin{align}\label{eq;psi g}
&J_f(\bx;\ba; \Psi ) 
=
\int_{U}
f(h)\varphi(h)\prod_{i=1}^{k-1}\psi _i( a_ihx_i)\dd h+\new {O_{f }(k\|\Psi\|_{\mathrm{sup}} ^{k-1}\|\Psi\| r)},
\end{align}
where 
\begin{align*}
\varphi(h)&=\int_{ G ^+}
\psi _k(b_t g b_{-t}a_khx_k) \,\theta(g)\dd g.
\end{align*}

We choose $r_0>0$ so that $\supp f\subset B_{r_0}^U$. Since $r_0$ is determined by $f$, any constant depending
on $r_0$  in fact depends on $f$. 
We take
\begin{align*}
B_1&=\{h\in B_{r_0}^U: b_{-t} a_khx_k\in \mathrm{Inj}_{3r}^X \}\quad \mbox{and}\quad 
B_2=B_{r_0}^U\setminus B_1.
\end{align*}
Since $b_{-t} a_k\in \overline{A_U^+}$, 
Theorem \ref{thm;general} implies 
\begin{align}\label{eq;B 2}
\mu_U(B_2)\ll_{ L, f} r^{\delta_1}.
\end{align}
By (\ref{eq;psi g}) and (\ref{eq;B 2})
\begin{align}
 J_f(\bx;\ba; \Psi ) \label{eq;newsub}
=&\left(\int_{B_1}+\int_{B_2}\right)
f(h)\varphi(h)\prod_{i=1}^{k-1}\psi _i( a_ihx_i)\dd h  +\new {O_{f}(k\| \Psi\|^{k-1}_{\mathrm{sup}}\|\Psi\| r) }     \\
=& \int_{B_1}f(h)\varphi(h)\prod_{i=1}^{k-1}\psi _i( a_ihx_i)\dd h+\new {O_{L,f}(k\|\Psi\|_{\mathrm{sup}}^{k-1}\|\Psi\| r^{\delta_1})}.\notag
\end{align}
For $h\in B_1$,
it follows from Lemma \ref{lem;km 12 general} and the properties of $\theta $ that 
\begin{align}\label{eq;psi k}
\left |\varphi(h)-\int_X\psi _k \dd \mu_X  \right |
\ll \new {( \|\psi_k\|_{\ell}+\|\psi_k\|_{\mathrm{Lip}})} ( r^{-3\ell} e^{-t\delta_0}+r).
\end{align}
By (\ref{eq;newsub}) and    (\ref{eq;psi k}), we have 
\begin{align*}
J_f(\bx;\ba; \Psi )  
&=\int_X\psi _k\dd \mu_X\cdot\int_U f(h)\prod_{i=1}^{k-1}\psi _i( a_ihx_i)\dd h\\
&+\new {O_{L,f}(k\|\Psi\|^{k-1}_{\mathrm{sup}} \|\Psi\|(r^{\delta_1}+ r^{-3\ell} e^{-t\delta_0}))}.
\end{align*}

In view of the choice of $\delta $ in (\ref{eq;effect delta}) and $r$ in (\ref{eq;effect r}) one has
\begin{align*}
J_f(\bx;\ba; \Psi )=\int_X\psi _k\dd \mu_X\cdot\int_U f(h)\prod_{i=1}^{k-1}\psi _i( a_ihx_i)\dd h
+\new {O_{L, f}(k\|\Psi\|^{k-1}_{\mathrm{sup}}\|\Psi\|e^{-t\delta })}.
\end{align*}

\end{proof}

\begin{proof}[Proof of Theorem \ref{thm;effective}]
 We will show that 	Theorem \ref{thm;effective} holds for $\delta $ in (\ref{eq;effect delta}) and 
	$M=M_0 k^2 \|\Psi\|_{\mathrm{sup}}^{k-1} \|\Psi\| $ where $M_0\ge 1$ is a constant  depending 
	only
	on $L$ and $f$. 
Using  (\ref{eq;prop new}) for each  $k, k-1, \ldots , 1$ and $J_f(\bx; \tilde \ba;\Psi)$ step by step, 
we get
\begin{align*}
I_f(\bx;\ba; \Psi )=\int_U f\cdot \prod_{i=1}^k\int_X\psi_i \dd \mu_X +\new{O_{L, f}(k^2\|\Psi\|_{\mathrm{sup}}^{k-1}\|\Psi\|e^{-\delta \lfloor \ba \rfloor})}. 
\end{align*}
\end{proof}

\begin{proof}
	[Proof of Theorem \ref{thm;periodic}]
     The proof is in principle the same as that of Theorem \ref{thm;effective} with
     $k=1, a_1=a , \psi_1=\psi$ and 
      $f$ equal to   the characteristic function of the fundamental domain $F$ of $Ux$. 
     So we only sketch the proof and 
     indicate the modifications. 
     
      We begin with  replacing  $\lfloor \cdot \rfloor$ by $\lfloor \cdot \rfloor'$. 
      Then in (\ref{eq;d 1gi}) we loose the  control of $d_{G^+}(1_G, g_0)$ if the conjugation of  $a_1'=b_{-t}a_1$ does not expand $g$.   
 Using a  different idea which is due to \cite{dkl}, instead of (\ref{eq;psi g}), we get   
     \begin{align}\label{eq;change}
     I_f(x; a; \psi)&= \int_U \theta (g)\int_{g_0^{-1} F}  \psi (b_t g a_1' h x )
     \dd h\dd g\\
      &= \int_U \theta (g)\int_{ F}  \psi (b_t g a_1' h x ) \dd h\dd g\notag \\
    & =\int_ { F} \int_U  \theta (g) \psi (b_t g a_1' h x )
     \dd g\dd h,\notag
     \end{align}
     where in the second equality we use the assumption that $F$ is a fundamental domain of the periodic orbit $Ux$. All the rest parts of the proof  are the same if we change the use of (\ref{eq;psi g}) by (\ref{eq;change}). 
\end{proof}

\begin{proof}[Proof of Corollary \ref{cor;effective}]

	It is a special case of Theorem \ref{thm;coninuous}. 
\end{proof}

\appendix

\section{Multiple ergodic theorem}
\begin{center}
		{Rene R\"uhr}\footnote{School of Mathematical Sciences, Tel Aviv University, Tel Aviv, Israel. R. R\"{u}hr was supported in part by the S.N.F. grant project number 168823. E-mail: reneruehr@googlemail.com}\quad and \quad
		{Ronggang Shi}
\end{center}

 The study of multiple pointwise ergodic theorem 
has a long history. The first result is double pointwise ergodic theorem due to Bourgain \cite{b}. Recently there are some progress on this question  for ergodic distal systems Huang-Shao-Ye \cite{hsy} and weakly mixing pairwise independent systems  Gutman-Huang-Shao-Ye \cite{ghsy}. We remark here that all these results are for different orders of a fixed measure preserving transformation. On the other hand, 
the multiple mean  ergodic theorem for commuting maps are well-understood, see Tao \cite{t} and references there.

In this appendix we prove a multiple pointwise ergodic theorem with an error rate on homogeneous space for commuting actions. Let the notation and  assumptions be as in Theorem \ref{thm;effective}. Recall that $\mathfrak a_{\mathfrak u}^+\subset \mathfrak a$ such that $\exp a_{\mathfrak u}^+=A^+_U$.
	Let 
\[
\mathfrak a^+_{\mathfrak u,   b}=\{s\in \mathfrak a_{\mathfrak u}^+: H_b^+\le  H_{\exp s}^+\cap U \}. 
\]
\begin{thm}
	\label{thm;coninuous}
	Let $s_1, \ldots, s_m\in \mathfrak a^+_{\mathfrak u, b}$, $w_i=s_1+\cdots +s_i$, $ \varphi_i\in 
	\mathcal S(X) $ and $y_i\in X$,  where $i=1, \ldots, m$ and $m\in \N $. 
	Then for any   $\varepsilon >0$ we have  
	\begin{align}\label{eq;multiple}
	\frac{1}{T}\int_0^T\prod_{i=1}^m \varphi_i(\exp (tw_i) h y_i)  \dd t =\prod_{i=1}^m\mu_X(\varphi_i)
	+o	(T^{-1/2}\log^{{\frac32+\varepsilon} } T)
	\end{align}
 for $\mu_U$ almost every $h\in U$. 
\end{thm}

The conclusion of the  theorem is called the  effective $m$-multiple ergodic theorem. 
It is already noticed in \cite{ksw} that the effective double equidistribution implies the  effective  ergodic theorem. Here we use the effective $2m$-step equidistribution to prove the  
effective $m$-multiple ergodic theorem. 

We are going to reduce the  multiple ergodic average to the  ergodic average and use the following theorem. 

	\begin{thm}[\cite{ksw} Theorem 3.1]\label{thm;ksw}
	Let $(Y, \nu)$  be a probability space, and let $F:Y\times \R_+\to \R$ be a bounded measurable function.
	Suppose  there exist $\delta >0$ and $C>0$ such that for any $r
	\ge t \ge 0$,
	\begin{align}\label{eq;effective rate}
	\left |
	\int_Y{F(x,t)}{F(x,r)}\dd \nu( x)
	\right|
	\le Ce^{-\delta \min(t, {r}-t)}.
	\end{align}
	Then for any  $\varepsilon >0$ we have 
	\begin{align*}
	\frac{1}{T}\int_0^T {F(y,t)}\dd t=o
	(T^{-1/2}\log^{{\frac32+\varepsilon} } T)
	\end{align*}
	for $\nu$ almost every ${y}\in Y$. 
\end{thm}
We will apply the above theorem with $Y=U$ and 
\begin{align}\label{eq;fut}
F(h, t)= \varphi_1(\exp (tw_1) h y_1)\cdots \varphi_m({\exp (tw_m) hy_m }).
\end{align}

\begin{lem}
	\label{lem;tech}
	Suppose that  $f\in C_c^\infty (U)$ and  $\mu_X(\varphi_m)=0$. 	Then there  exist $\sigma >0$ and $C>0$ such that for any $r
	\ge t \ge 0$,
	\begin{align}
	\left |
	\int_U{F(h,t)}{F(h,r)}f(h)\dd \mu_U( h)
	\right|
	\le Ce^{-\sigma( {r}-t)}.
	\end{align}
\end{lem}

\begin{proof}
We apply  Theorem \ref{prop;explicit} with $k=2m$ and 
	\begin{align*}
&\psi_i=\varphi_i,\quad  \psi_{m+i}=\varphi_i, \\
& a_i=\exp(tw_i),\quad  a_{m+i}=\exp(r w_i), \\
& x_i=y_i , \quad x_{m+i}=y_i\quad (1\le i\le m). 
\end{align*}
There exists $\kappa >0$ depending on $w_i\ (1\le i\le m)$ and $b$ such that 
\[
\min_{i=0}^k\lfloor a_k a_i^{-1} \rfloor\ge \kappa ( r-t). 
\]
Let $\delta$ be as in Theorem \ref{prop;explicit}
and $\sigma=\kappa \delta$. 
Then it follows from 
	(\ref{eq;prop new}) 
	and the assumption $\mu_X(\varphi_m)=0$  that 
	\begin{align*}
	\left |
\int_U{F(h,t)}{F(h,r)} f(u)\dd \mu_U( h)
\right|= |	J_f(\bx;\ba; \Psi)|\ll e^{-\sigma(r-t)}, 
	\end{align*}
	where the implied constant depends on $\bx, f$ and $ \Psi $. 
\end{proof}

\begin{proof}[Proof of Theorem \ref{thm;coninuous}]
	We prove it  by induction on $m$. Assume the theorem holds for $m-1$, where if $m-1=0$ we mean it 
	holds for the  constant function. 
	The left hand side of (\ref{eq;multiple}) is equal to 
	\begin{align*}
	\frac{1}{T}\int_0^T(\varphi_m-\mu_X(\varphi_m))\prod_{i=1}^{m-1} \varphi_i(\exp (tw_i) h x_i)  \dd t \\
	+\mu_X(\varphi_m)\cdot \frac{1}{T}\int_0^T\prod_{i=1}^{m-1} \varphi_i(\exp (tw_i) h x_i)  \dd t.
	\end{align*}
	In view of the induction hypothesis, it suffices  to prove  (\ref{eq;multiple}) under the additional assumption that  $\mu_X(\varphi_m)=0$.

	Assume now that $\mu_X(\varphi_m)=0$.	
		Let $f\in C_c^\infty (U)$ with $f\ge 0$ and $\int_U f\dd \mu_U=1$. Let $\nu=f(h)\dd \mu_U(h)$ be the associated probability measure on $U$. 
	It follows from Lemma \ref{lem;tech} 
      that the assumption of Theorem \ref{thm;ksw} holds for the function $F(h, t)$ defined by (\ref{eq;fut}) and the measure $
      \nu$ on $U$.
Therefore (\ref{eq;multiple}) holds for $\mu_U$ almost every $h\in \supp f$. By taking a countable 
	family $\{f_n\}$  of such $f$  with   $\bigcup \supp f_n=U$ one gets that 
	(\ref{eq;multiple}) holds for $\mu_U$ almost every $h$. 
\end{proof}

\end{document}